\documentclass[11pt]{article}

\usepackage{amsmath,amsthm,verbatim,amssymb,amsfonts,amscd, graphicx, mathrsfs, mathtools, hyperref, cleveref, dsfont, nomencl, xcolor}
\usepackage{graphics}
\usepackage{tikz}
\usepackage{tikz-cd}
\usetikzlibrary{matrix,arrows,decorations.pathmorphing}
\usepackage[thicklines]{easytable}
\numberwithin{equation}{section}

\theoremstyle{plain}
\newtheorem{theorem}{Theorem}[section]

\newtheorem{lemma}[theorem]{Lemma}
\newtheorem{proposition}[theorem]{Proposition}

\theoremstyle{definition}
\newtheorem{definition}[theorem]{Definition}
\newtheorem{remark}{Remark}[section]

\hypersetup{
    colorlinks=true, 
    linkcolor=blue, 
    urlcolor=red, 
    linktoc=all 
    }

\newcommand{\1}{\mathds{1}}
\newcommand{\C}{\mathcal{C}}

\newcommand {\CC}{\mathbb{C}}
\newcommand{\Rep}{\text{Rep}}
\newcommand{\F}{\mathcal{F}}
\newcommand{\ind}{\mathscr{F}}

\newcommand{\ZZ}{\mathbb{Z}} 
\newcommand{\CCC}{\mathcal{C}} 
\newcommand{\dCCC}{\mathcal{C}_\oplus} 
\DeclareMathOperator{\ob}{\text{Ob}} 
\DeclareMathOperator{\im}{im} 
\DeclareMathOperator{\Id}{Id} 
\DeclareMathOperator{\Span}{Span} 
\DeclareMathOperator{\Hom}{Hom} 
\DeclareMathOperator{\inj}{inj} 

\newlength{\domainlength}
\newlength{\codomainlength}
\newlength{\typicaleltlength}
\newlength{\imagelength}

\DeclareMathOperator{\shift}{Shift} 

\title{On Infinite Order Simple Current Extensions \\ of Vertex Operator Algebras}
\author{Jean Auger, Matt Rupert}
\date{}
\begin{document}
\maketitle

\begin{abstract} 
We construct a direct sum completion $\dCCC$ of a given braided monoidal category $\C$ which allows for the rigorous treatment of infinite order simple current extensions of vertex operator algebras as seen in \cite{CKL}. As an example, we construct the vertex operator algebra $V_L$ associated to an even lattice $L$ as an infinite order simple current extension of the Heisenberg VOA and recover the structure of its module category through categorical considerations.
\end{abstract}

\tableofcontents
\bigskip

\section{Introduction} \label{intro}

Many interesting vertex operator algebras (VOAs) such as the $\mathcal{B}_p$-algebras (see \cite{C,CRW,ACKR}) motivated by a study of $L_{k}(\mathfrak{sl}_2)$ \cite{A}, the logarithmic parafermion algebras of \cite{ACR}, and $L_{-1}\big(\mathfrak{sl}(m|n)\big)$ of \cite{CKLR,KW}, among others, can be realised as infinite order simple current extensions. The triplet VOA $\mathcal{W}(p)$ may also be realized as an infinite order simple current extension of the singlet VOA $\mathcal{M}(p)$ and a parallel construction may also be applied to Hopf algebras, as seen in \cite{CGR}.  Simple current extensions have also appeared in the study of conformal embeddings \cite{AKMPP,KMPX}. A simple current (see Definition \ref{simplecurrent}) $J$ in the module category $\C$ of a VOA $V$ satisfying certain conditions (see \cite[Theorem 3.12]{CKL}) can be used to construct a VOA extension:

\begin{equation} \label{simplecurrentext}
V_e = \bigoplus\limits_{n \in G} J^{\otimes n}
\end{equation}
where $G = \ZZ / n\ZZ$ when $n=\text{ord}(J)$ is finite and $G=\ZZ$ when ord($J)=\infty$, and $J^0 = V = \mathds{1}_\C$. The extended VOA $V_e$ is an algebra object in $\C$ called a \textit{simple current extension} of $V$. The theory of algebra objects was developed in \cite{KO} with its applications to VOAs indicated therein. These applications were made rigorous in \cite{HKL}. It was shown in \cite{HKL} that if $\C$ is a vertex tensor category (see \cite{HLZ1}-\cite{HLZ8}), then the existence of the extension $V_e$ is equivalent to the existence of a \textit{haploid algebra object} in $\C$. It was also shown in \cite{HKL} that when $\C$ is a vertex tensor category, the category of generalised modules $\text{Mod}^GV_e$ of the extension VOA $V_e$ is equivalent as an abelian category to a category of modules $\Rep^0 V_e$ where $V_e$ is seen as an algebra object in $\C$. This was generalised to superalgebra objects in \cite{CKL}, and the equivalence was shown to be a braided monoidal equivalence in \cite{CKM}. It is possible to determine much of the categorical structure of $\text{Rep}^0V_e$ from $\C$ through the techniques provided in \cite{CKM}. Other notable results on simple current extensions can be found in \cite{DLM,FRS,La,LaLaY,Y}. \\

When $J$ is an infinite order simple current ($\#G = \infty$), the object $V_e$ is an infinite direct sum and therefore is not in general an object of $\C$. This poses a problem because the theorems of \cite{CKL} assume that $V_e$ is an algebra object of $\C$. One solution to this problem is to introduce a second grading on the modules as in \cite{HLZ1}-\cite{HLZ8}. Here, however, we appeal to a certain completion of $\C$ which allows for infinite direct sums. A well known candidate for such a completion is the Ind-category Ind($\C$) constructed in \cite{AGV}. The category Ind($\C$) is a natural completion of $\C$ under general \textit{inductive limits} and is both larger and more sophisticated than we require for the study of infinite order extensions. Since the machinery of the Ind-category is quite abstract, we prefer to focus on a direct sum completion $\dCCC$ of $\C$ rather than the full $\text{Ind}(\C)$.  \\

The purpose of this paper is two-fold. Firstly, we construct a direct sum completion $\dCCC$ of $\C$ 
that complements \cite{CKL} for the rigorous study of infinite order simple current extensions. Secondly, we illustrate the power of \cite{CKM} in the simplest possible example, realising the VOA $V_L$ associated to a rank 1 even lattice $L$ as a simple current extension of the Heisenberg VOA $\mathcal{H}$. The techniques of \cite{CKM} allow us to determine the structure of $\Rep^0 V_L$ and show that it coincides with the structure of Mod$V_L$ found in \cite{D,DL,LL}. Recent work on Drinfeld categories found in \cite{DF} can be applied to describe Mod$V_L$ through algebra objects, however our approach is more applicable to infinite order simple current extensions in general. The present work also paves the way for the study of richer examples of infinite order simple current VOA extensions such as the aforementioned $\mathcal{B}_p$-algebras, logarithmic parafermion algebras, and $L_{-1}\big(\mathfrak{sl}(m|n)\big)$.  \\

After a brief review of pertinent categorical notions in Section~\ref{sec:background}, we detail the completion category $\dCCC$ in Section \ref{sec:category}. We finish by analysing the infinite order simple current extension $\mathcal{H} \subset V_L$ in Section~\ref{sec:application} using the methods of \cite{CKM} within the framework of Section~\ref{sec:category}.

\section{Background} \label{sec:background}

In this section we will recall some of the fundamental concepts to which we will refer in subsequent sections including categories, algebra objects within categories, and simple currents.

\subsection{Category theory}

Recall that a category $\C$ is a class of objects $\ob(\C)$ and a class of morphisms $\Hom_{\C}(U,V)$ for each pair $U,V \in \text{Ob}(\C)$ together with an associative binary operation called composition 
$$ \circ : \; \Hom_{\C}(V,W) \times \Hom_{\C}(U,V) \rightarrow \Hom_{\C}(U,W) \; , $$ 
for which each object $V \in \ob(\C)$ has an \textit{identity} morphism $\Id_V \in \Hom_\C(V,V)$ that preserve any morphisms they are composed with.

\begin{definition}
A category $\mathcal{C}$ is \textit{additive} if
\begin{itemize}
\item $\Hom_{\C}(U,V)$ is an abelian group for every pair of objects $U,V \in \ob(\C)$ and composition of morphisms is bi-additive,
\item $\C$ has a a zero object $0$ such that $\Hom_{\C}(0,0) = 0$ is the trivial abelian group,
\item $\C$ contains finite direct sums (finite coproducts). That is, for every pair of objects $V_1,V_2 \in \ob(\C)$, there exists $W = V_1 \oplus V_2 \in \ob(\C)$ and morphisms $p_1: W \rightarrow V_1$, $p_2 : W \rightarrow V_2$, $i_1 : V_1 \rightarrow W$, $i_2 : V_2 \rightarrow W$ such that $p_1 \circ i_1 = \Id_{V_1}$, $p_2 \circ i_2 = \Id_{V_2}$, and $i_1 \circ p_1+i_2 \circ p_2 = \Id_W$.
\end{itemize}

Given a field $\mathbb{F}$, an abelian category $\C$ is called \textit{$\mathbb{F}$-linear} if for each $U,V \in \ob(\C)$, $\Hom_\C(U,V)$ is a vector space over $\mathbb{F}$ and the composition is $\mathbb{F}$-bilinear.
\end{definition}

%
%
%
%

A \textit{tensor product} on a category is a bifunctor $\otimes: \C \times \C \rightarrow \C$ that commutes with finite direct sums.
%
%
An \textit{associativity constraint} $a$ on $\mathcal{C}$ is a family $\big\{ a_{U,V,W}:(U \otimes V) \otimes W \to U \otimes (V \otimes W) \big\}_{U,V,W \in \ob(\C)}$ of natural isomorphisms. An associativity constraint satisfies the \textit{pentagon axiom} if the diagram

\begin{center}
\begin{tikzcd}
\left( (U \otimes V) \otimes W \right) \otimes X  \arrow{dd}{a_{U,V \otimes W,X}} &\;& \left( (U \otimes V) \otimes W \right) \otimes X \arrow{ll}{a_{U,V,W} \otimes \Id_X} \arrow{dr}{a_{U \otimes V,W,X}} &\\
&\;&\;& (U \otimes V) \otimes (W \otimes X) \arrow{dl}{a_{U,V,W \otimes X}}\\
U \otimes \left( (V \otimes W) \otimes X \right) \arrow{rr}[swap]{\Id_U \otimes a_{V,W,X}} &\;& U \otimes \left(V \otimes (W \otimes X) \right) & \;
\end{tikzcd}
\end{center}
 commutes for every choice of objects $U,V,W,X \in \ob(\C)$. A \textit{left unit constraint} $l$ in $\mathcal{C}$ with respect to an object $\1 \in \text{Ob}(\C)$ is a family $\big\{ l_V: \mathds{1} \otimes V \to V \big\}_{V \in \ob(\C)}$ of natural isomorphisms. \textit{Right unit constraints} $\big\{r_V:V \otimes \mathds{1} \to V\big\}_{V \in \ob(\C)}$ are defined similarly. The associativity, left, and right constraints satisfy the \textit{triangle axiom} if the diagram
%
\begin{equation*}
\begin{tikzcd}
(U \otimes \1) \otimes V \arrow{rr}{a_{U,I,V}}\arrow{rd}[swap]{r_U \otimes \Id_V} &&U \otimes (\1 \otimes V) \arrow{dl}{\Id_U \otimes l_V}\\
& U \otimes V
\end{tikzcd}
\end{equation*}
commutes for every pair of objects $U,V \in \text{Ob}(\mathcal{C})$. A triple $(\1,l,r)$ is called a \textit{unit} in $\C$ if $l$ and $r$ are left and right unit constraints with respect to $\1$, respectively, that satisfy the triangle axiom.

\begin{definition}
A \textit{monoidal category} $(\C,\otimes,a,\mathds{1},r,l)$ is a category $\mathcal{C}$ equipped with a tensor product $\otimes$, associativity constraint $a$ satisfying the pentagon axiom, and a unit object $\mathds{1}$ with left and right unit constraints $l,r$ satisfying the triangle axiom. 
\end{definition}

A \textit{commutativity constraint} $c$ on $\C$ is a family $\big\{ c_{U,V}:U \otimes V \to V \otimes U\big\}_{U,V \in \ob(\C)}$ of natural isomorphisms. A \textit{braiding} is a commutativity constraint which also satisfies the \textit{hexagon axiom}, which is the commutativity of the diagram
%
\begin{center}
\begin{tikzcd}
\; &  U \otimes (V \otimes W) \arrow{rr}[swap]{c_{U, V \otimes W}}  &\;& (V \otimes W) \otimes U \arrow{dr}{a_{V,W,U}}& \; \\
  (U \otimes V) \otimes W\arrow{ur}[swap]{a_{U,V,W}} \arrow{dr}{c_{U,V} \otimes \Id_W} &\;&\;&\;& V \otimes (W \otimes U)\\ 
\; & (V \otimes U) \otimes W \arrow{rr}[swap]{a_{V,U,W}}  &\;&V \otimes (U \otimes W) \arrow{ur}[swap]{\Id_V \otimes c_{U,W}}&\;
\end{tikzcd}
\end{center}
and of the analagous diagram for $a^{-1}$.
\begin{definition}
A braided monoidal category is a monoidal category with a braiding $c$. 
\end{definition}

A twist $\theta$ in a braided monoidal category $\C$ is a family $\{\theta_V:V \rightarrow V\}_{V \in \ob(\C)}$ of natural isomorphisms such that the \textit{balancing axiom}
\[ \theta_{U \otimes V}=c_{V,U}\circ c_{U,V} \circ(\theta_U \otimes \theta_V) \; ,\]
holds. Let $V \in \ob(\C)$ and suppose there is an associated object $V^*$ with \textit{duality morphisms}

\[ \overrightarrow{\text{coev}}_V:\mathds{1} \rightarrow V \otimes V^*, \qquad \overrightarrow{\text{ev}}_V:V^* \otimes V \rightarrow \mathds{1} \; . \]
which satisfy the relations
\begin{align*}
r_V \circ (\Id_V \otimes \overrightarrow{\text{ev}}_V)\circ a_{V,V^*,V} \circ(\overrightarrow{\text{coev}}_V \otimes \Id_V) \circ l_V^{-1}&=\Id_V,\\
l_{V^*} \circ (\overrightarrow{\text{ev}}_V \otimes \Id_{V^*})\circ a^{-1}_{V^*,V,V^*} \circ (\Id_{V^*} \otimes \overrightarrow{\text{coev}}_V) \circ r_{V^*}^{-1}&= \Id_{V^*}.
\end{align*}
%
Then, $V^*$ is said to be \textit{left dual} to $V$. If a left dual exists for every $V \in \ob(\C)$ then $\mathcal{C}$ is called left rigid. The duality morphisms are said to be compatible with the braiding and twist if they satisfy the relation

\[ (\theta_V \otimes \Id_{V^*}) \circ \overrightarrow{\text{coev}}_V=(\Id_V \otimes \theta_{V^*}) \circ \overrightarrow{\text{coev}}_V \; . \]

\begin{definition}
A \textit{ribbon category} is a left rigid braided monoidal category with twist and compatible duality.
\end{definition}

\subsection{Simple currents and algebra objects}

Let $\mathbb{F}$ be a field and $\CCC$ be a $\mathbb{F}$-linear braided monoidal category with twist $\theta$. 

\begin{definition} \label{simplecurrent}
A simple current in $\CCC$ is a simple object $J$ which is invertible with respect to the tensor product. That is, there exists an object $J^{-1} \in \C$ satisfying $J \otimes J^{-1} \cong \mathds{1}$.
\end{definition}

\begin{remark}
When $\C$ is a module category for a simple vertex operator algebra $V$, the above definition for simple currents is equivalent to the duality morphisms being isomorphisms, which is Definition 2.11.1 for invertibility given in \cite{EGNO}.
\end{remark}

As noted in the introduction, it is natural to expect that a simple current extension $V \subset V_e$ as of \eqref{simplecurrentext} and its representation theory can be related to $V$ by categorical means. In fact, it can be shown (see \cite{CKM}) that the category of generalised modules of the VOA $V_e$ is braided equivalent to a category $\Rep^0 V_e$ defined below, where $V_e$ is seen as an \textit{algebra object} in the category $\C$ rather than a VOA.  These ideas also appeared in \cite{CKL,KO}. \\

For the rest of the subsection, we recall the key notions of algebra objects and of modules for algebra objects. 

\begin{definition} \label{algebraobjectdef}
An \textit{associative unital and commutative algebra} in the category $\C$ is a triple $(A,\mu,\iota_A)$ where $A \in \ob(\C)$, $\mu \in \Hom_\C(A \otimes A, A)$ and $\iota_A \in \Hom_\C(\mathds{1},A)$ are subject to the following assumptions:
\begin{itemize}
\item Associativity: $\mu \circ ( \Id_A \otimes \mu) = \mu \circ (\mu \otimes \Id_A) \circ a^{-1}_{A,A,A} : A \otimes (A \otimes A) \rightarrow A \;$;
\item Unit: $\mu \circ (\iota_A \otimes \Id_A) \circ l_A^{-1} = \Id_A : A \rightarrow A \;$; 
\item Commutativity: $\mu \circ c_{A,A} = \mu : A \otimes A \rightarrow A \;$. 
\end{itemize}
\end{definition}


\begin{definition} \label{RepAdef} Let $(A,\mu,\iota_A)$ be an associative unital and commutative algebra object in $\C$. Define $\Rep A$ to be the category whose objects are given by pairs $(V,\mu_V)$ where $V \in \ob(\mathcal{C})$ and $\mu_V \in \Hom_\CCC(A \otimes V,V)$ are subject to the following assumptions:
\begin{itemize}
\item Associativity: $\mu_V \circ (\Id_A \otimes \mu_V)=\mu_V \circ ( \mu \otimes \Id_A) \circ a^{-1}_{A,A,V} : A \otimes ( A \otimes V) \rightarrow V \;$; 
\item Unit: $\mu_V \circ (\iota_A \circ \mathds{1}) \circ l_V^{-1} = \Id_V : V \rightarrow V$.
\end{itemize}
The morphisms of $\Rep A$ are defined as follows: 
$$ \Hom_{\Rep A}\big((V,\mu_V),(W,\mu_W)\big) = \left\{ f \in \Hom_{\CCC}\left(V,W\right) \; \big| \; f \circ \mu_V = \mu_W \circ (\Id_A \otimes f) \right\} \; . $$
Thus, a morphism in $\Rep A$ is just a morphism in $\CCC$ that intertwines the $A$-actions maps. An object of $\Rep A$ is also called an $A$-module.
\end{definition}

Given that the base category $\C$ is monoidal, one can define a new tensor product, $\otimes_A$, that makes $\Rep A$ a monoidal category as well. For more details, see \cite{EGNO,KO,P}. This new tensor product is defined as follows: 

\begin{definition} \label{tensorAdef} Let $(V,\mu_V),(W,\mu_W) \in \ob(\Rep A)$. Define their tensor product to be the pair $(V \otimes_A W, \mu_{V \otimes_A W})$ where
\begin{equation} \label{tensorAobject}
V \otimes_A W = \frac{V \otimes W}{\im (m^{\text{left}} - m^{\text{right}})} \; ,
\end{equation}
and
\begin{align*}
m^{\text{left}} = (\mu_V \otimes \Id_W) \circ \, (c_{V,A} \otimes \Id_W) \, \circ a^{-1}_{V,A,W} : \quad &V \otimes (A \otimes W) \longrightarrow V \otimes W \; , \\
m^{\text{right}} = \Id_V \otimes \mu_W : \quad &V \otimes (A \otimes W) \longrightarrow V \otimes W \; .
\end{align*}
Note that in the quotient \eqref{tensorAobject}, the action of $A$ on $V \otimes W$ via $\mu_V$ is identified with its action via $\mu_W$. Hence, one can simply define 
$$ \mu_{V \otimes_A W} = \mu_V \otimes \Id_W \, \circ \; a^{-1}_{A,V,W} : A \otimes (V \otimes_A W)  \to V \otimes_A W\; , $$ 
so that $(V,\mu_V) \otimes_A (W,\mu_W) = (V \otimes_A W,\mu_{V \otimes_A W}) \in \ob(\Rep A)$.
\end{definition}

In general, $\Rep A$ may not be braided, however, it was proven in \cite{P} that the braiding of $\C$ induced a braiding on a full subcategory of $\Rep^0A$ of $\Rep A$ defined as follows:

\begin{definition} \label{Rep0def} Let $(A,\mu,\iota_A)$ be an associative unital and commutative algebra object in $\C$. Define $\text{Rep}^0A$ to be the full subcategory of $\Rep A$ whose objects $(V,\mu_V)$ satisfy
$$ \mu_V \circ ( c_{V,A} \circ c_{A,V} ) =\mu_V \; . $$
The category $\Rep^0 A$ is often referred to as the category of \textit{local} or \textit{untwisted} $A$-modules.
\end{definition}



A valuable tool for the study of both $\Rep A$ and $\Rep^0 A$ is the following induction functor:
\begin{definition} \label{indfunctordef} Let $(A,\mu,\iota_A)$ be an associative unital and commutative algebra object in $\C$. Define a functor $\ind :  \CCC  \longrightarrow \Rep A$ by
\begin{align*}
V & \longmapsto \big(A \otimes V \, , \, (\mu \otimes Id_V) \circ a_{A,A,V}^{-1}\big) \; ,\\
[V \overset{f}{\rightarrow} W] & \longmapsto \Id_A \otimes f \; .
\end{align*}
\end{definition}

A crucial property of this induction functor $\ind$ is that it is a \textit{monoidal functor} (a functor that respects tensor products up to fixed natural isomorphisms). In Section 2 of \cite{CKM}, the authors study further properties of $\ind$ and obtain the following result:

\begin{proposition} \label{C0induction}{\em {\bf \cite[Theorem 2.67]{CKM} }} Let $\C^0$ denote the full subcategory of $\C$ consisting of objects that induce to $\Rep^0 A$. Then $\ind : \C^0 \rightarrow \Rep^0 A$ is a braided tensor functor.
\end{proposition}

\section{Sum Completion of a Category $\CCC$} \label{sec:category}

In this section, we construct a direct sum completion $\C_{\oplus}$ of an additive $\mathbb{F}$-linear category $\C$ and show how to transfer additional categorical structure from $\C$ to $\dCCC$. \\


%


%
\subsection{The category $\dCCC$}

In this subsection, let $\CCC$ be an additive $\mathbb{F}$-linear category. The category $\dCCC$ can be thought of as the subcategory of $\text{Ind}(\CCC)$ \cite{AGV} whose objects are the inductive systems that produce arbitrary \textit{coproducts}, see \cite{PP} for instance.

\begin{definition} \label{sumcomp} Define $\dCCC$ by setting:
\begin{align*}
\ob(\dCCC) = \left\{ \; \bigoplus_{s \in S} X_s \; \left| \; \begin{array}{l} S \text{ is a set} \\ X_s \in \ob(\CCC) \text{ for every } s \in S \end{array}\right. \right\} \\
\Hom_{\dCCC}\left(\;\bigoplus_{s \in S} X_s,\bigoplus_{t \in T} Y_t\right) = \left\{\Big(\alpha,\{f_{s,t}\}_{s \in S}^{t \in \alpha(s)}\Big)\right\} \Big/ {}_\sim
\end{align*}
where
\begin{itemize}
\item $\alpha : \{\text{Finite subsets of } S\} \rightarrow \{\text{Finite subsets of } T\}$ is a function that commutes with unions. 
For any singleton $\{s\} \subseteq S$, we let $\alpha(s) = \alpha\big(\{s\}\big)$;
\item $f_{s,t} \in \Hom_\CCC(X_s,Y_t)$ for any $t \in \alpha(s)$;
\item $\sim$ is an equivalence relation defined by:
\begin{equation*} 
\Big(\alpha,\{f_{s,t}\}_{s \in S}^{t \in \alpha(s)}\Big) \sim \Big(\beta,\{g_{s,t}\}_{s \in S}^{t \in A(s)}\Big) \quad \Leftrightarrow \quad 
\left[\begin{array}{cl}
(1) \quad f_{s,t} = 0_{s,t} & \text{if } t \in \alpha(s) \backslash \beta(s) \\
(2) \quad f_{s,t} = g_{s,t} & \text{if } t \in \alpha(s) \cap \beta(s) \\
(3) \quad g_{s,t} = 0_{s,t} & \text{if } t \in \beta(s) \backslash \alpha(s)
\end{array}\right. ;
\end{equation*}
\item the composition of a pair of morphisms $\Big(\beta,\{g_{t,r}\}_{t \in T}^{r \in \beta(t)}\Big) \in \Hom_{\dCCC}(\bigoplus_{t \in T} Y_t , \bigoplus_{r \in R} Z_r)$ and $\Big(\alpha,\{f_{s,t}\}_{s \in S}^{t \in \alpha(s)}\Big) \in \Hom_{\dCCC}( \bigoplus_{s \in S} X_s , \bigoplus_{t \in T} Y_t)$ is defined to be the equivalence class of
\begin{align*}
\left(\beta \circ \alpha, \; \Bigg\{\sum_{\substack{t \in \alpha(s) \\ \text{ s.t. } r \in \beta(t)}} g_{t,r} \circ f_{s,t}\Bigg\}_{s \in S}^{r \in \big(\beta \, \circ \, \alpha\big)(s)}\right);
\end{align*}
\item the identity morphism of $\bigoplus_{s \in S} X_s$ is the equivalence class of $\left(\Id_{\text{f.s.}(S)},\{\Id_{X_s}\}_{s \in S}\right)$ where $\text{f.s.}(S)$ denotes the collection of finite subsets of $S$.
\end{itemize}
\end{definition}

\begin{proposition} The elements of Definition~\ref{sumcomp} define a category structure $\dCCC$. In particular:
\begin{itemize}
\item $\sim$ is an equivalence relation;
\item the composition is compatible with $\sim$ in both arguments;
\item the composition is associative;
\item the identity morphism of an object preserves any morphism under composition on both sides.
\end{itemize}
\begin{proof} First, let's show that $\sim$ is an equivalence relation. For reflexivity, just note that given a morphism $\Big(\alpha,\{f_{s,t}\}_{s \in S}^{t \in \alpha(s)}\Big)$ and a singleton $s \in S$, one has $\alpha(s) \backslash \alpha(s) = \emptyset$. Symmetry of the relation $\sim$ is clear from the definition. For transitivity, let 
\begin{equation} \label{transitivity}
\Big(\alpha,\{f_{s,t}\}_{s \in S}^{t \in \alpha(s)}\Big) \sim \Big(\beta,\{g_{s,t}\}_{s \in S}^{t \in \beta(s)}\Big) \quad \text{ and } \quad \Big(\beta,\{g_{s,t}\}_{s \in S}^{t \in \beta(s)}\Big) \sim \Big(\gamma,\{h_{s,t}\}_{s \in S}^{t \in \beta(s)}\Big), 
\end{equation}
be three morphisms between the same two objects of $\dCCC$. Fix $s \in S$ and let $t \in \alpha(s) \cap \gamma(s)$. Then $t$ is either in $\beta(s)$ or not. If it is, then $f_{s,t} = g_{s,t}$ by the first relation of~\eqref{transitivity} and $g_{s,t} = h_{s,t}$ by the second relation so that $f_{s,t} = h_{s,t}$ for such $s$ and $t$. Next, let $t \in \alpha(s) \backslash \gamma(s)$. If $t$ is also in $\beta(s)$, we get $f_{s,t} = g_{s,t}$ from the first relation, however such a $t$ has to be in $\beta(s) \backslash \gamma(s)$ so $g_{s,t} = 0_{s,t}$ by the second relation of~\eqref{transitivity}. It follows that $f_{s,t} = 0_{s,t}$ as desired. Else, $t$ is not in $\beta(s)$ so the first relation of~\eqref{transitivity} directly gives $f_{s,t} = 0_{s,t}$. Finally, reversing the roles of $\alpha$ and $\gamma$ in the previous argument gives $h_{s,t} = 0_{s,t}$ for $t \in \gamma(s) \backslash \alpha(s)$ and we conclude that $\sim$ is transitive, hence an equivalence relation.\\

Let's now show that the composition of morphisms is compatible with $\sim$ in both arguments. Let
\begin{align}
\Big(\alpha,\{f_{s,t}\}_{s \in S}^{t \in \alpha(s)}\Big) &\sim \Big(\tilde{\alpha},\{\tilde{f}_{s,t}\}_{s \in S}^{t \in \tilde{\alpha}(s)}\Big) \in \Hom_{\dCCC}\left(\;\bigoplus_{s \in S} X_s,\bigoplus_{t \in T} Y_t\right) , \label{firstequiv} \\
\Big(\beta,\{g_{t,r}\}_{t \in T}^{r \in \beta(t)}\Big) &\sim \Big(\tilde{\beta},\{\tilde{g}_{t,r}\}_{t \in T}^{r \in \tilde{\beta}(t)}\Big) \in \Hom_{\dCCC}\left(\;\bigoplus_{t \in T} Y_t,\bigoplus_{r \in R} Z_r\right) . \label{secondequiv}
\end{align}
Then one must show that 
$$ \left(\beta \circ \alpha,\left\{\sum_{{\substack{t \in \alpha(s) \\ \text{ s.t. } r \in \beta(t)}}} g_{t,r} \circ f_{s,t}\right\}_{s \in S}^{r \in \beta\big(\alpha(s)\big)}\right) \quad \sim \quad \left(\tilde{\beta} \circ \tilde{\alpha},\left\{\sum_{{\substack{t \in \tilde{\alpha}(s) \\ \text{ s.t. } r \in \tilde{\beta}(t)}}} \tilde{g}_{t,r} \circ \tilde{f}_{s,t}\right\}_{s \in S}^{r \in \tilde{\beta}\big(\tilde{\alpha}(s)\big)}\right) \; . $$
To do so, fix $s \in S$ and consider three cases for $r$:
\begin{enumerate}
\item[(1)] $r \in \beta\big(\alpha(s)\big) \cap \tilde{\beta}\big(\tilde{\alpha}(s)\big)$. In this case, the equivalence~\eqref{firstequiv} implies that for any $t \in \big(\alpha(s) \backslash \tilde{\alpha}(s)\big)$, $f_{s,t} = 0_{s,t}$ and for any $t \in \big(\tilde{\alpha}(s) \backslash \alpha(s)\big)$, $\tilde{f}_{s,t} = 0_{s,t}$. Additionally, $t \in \alpha(s) \cap \tilde{\alpha}(s)$ implies that $f_{s,t} = \tilde{f}_{s,t}$.
\begin{align}
\sum_{{\substack{t \in \alpha(s) \\ \text{ s.t. } r \in \beta(t)}}} g_{t,r} \circ f_{s,t} &= \sum_{{\substack{t \in \tilde{\alpha}(s) \\ \text{ s.t. } r \in \tilde{\beta}(t)}}} \tilde{g}_{t,r} \circ \tilde{f}_{s,t} \notag \\
\Longleftrightarrow \qquad \sum_{{\substack{t \in \alpha(s) \cap \tilde{\alpha}(s) \\ \text{ s.t. } r \in \beta(t)}}} g_{t,r} \circ f_{s,t} &= \sum_{{\substack{t \in \alpha(s) \cap \tilde{\alpha}(s) \\ \text{ s.t. } r \in \tilde{\beta}(t)}}} \tilde{g}_{t,r} \circ f_{s,t} \; . \label{sumscompositionequiv}
\end{align}
By the equivalence~\eqref{secondequiv}, the two sums of line~\eqref{sumscompositionequiv} must coincide. 
\item[(2)] $r \in \beta\big(\alpha(s)\big) \backslash \, \tilde{\beta}\big(\tilde{\alpha}(s)\big)$. As in the previous case, both sums will only display non-zero terms with an index $t \in \alpha(s) \cap \tilde{\alpha}(s)$. The choice of $r$ makes it impossible for such a $t$ to be in $\tilde{\beta}(t)$ for we have $\{t\} \subseteq \tilde{\alpha}(s) \, \Rightarrow \, \tilde{\beta}(t) \subseteq \tilde{\beta}\big(\tilde{\alpha}(s)\big)$. It follows that the right-hand sum of~\eqref{sumscompositionequiv} is empty and corresponds to $0_{s,r} \in \Hom_\CCC(X_s,Z_r)$. Finally, as any $t \in \alpha(s) \cap \tilde{\alpha}(s)$ such that $r \in \beta(t)$ are also such that $r \in \beta(t) \backslash \tilde{\beta}(t)$, all the terms of the left-hand sum of~\eqref{sumscompositionequiv} are zero by the equivalence~\eqref{secondequiv}. Hence, the two sums of~\eqref{sumscompositionequiv} match.
\item[(3)] $r \in \beta\big(\alpha(s)\big) \backslash \, \tilde{\beta}\big(\tilde{\alpha}(s)\big)$. This case is treated exactly as the case (2) above.
\end{enumerate}
In conclusion, the composition of morphisms is compatible with $\sim$ in both arguments. In particular, we can assume, without loss of generality, that an arbitrary non-zero morphism $\Big(\alpha,\{f_{s,t}\}_{s \in S}^{t \in \alpha(s)}\Big)$ satisfies $f_{s,t} \neq 0_{s,t}$ whenever defined. Moreover, such a reduced form has to be unique by the definition of $\sim$. Systematically reducing morphisms in such a way allows to view arbitrary compositions in a simpler way. Consider the component morphism of the composition $\Big(\beta,\{g_{t,r}\}_{t \in T}^{r \in \beta(t)}\Big) \circ \Big(\alpha,\{f_{s,t}\}_{s \in S}^{t \in \alpha(s)}\Big)$ corresponding to $X_s \rightarrow Z_r$ for fixed $s \in S$ and $r \in R$. We have the following natural bijection:
\begin{equation} \label{piecewisenonzero}
\left\{\begin{array}{c} \text{Terms of the component map} \\ \sum_{{\substack{t \in \alpha(s) \\ \text{ s.t. } r \in \beta(t)}}} g_{t,r} \circ f_{s,t} \end{array}\right\}
\qquad \stackrel{1:1}{\longleftrightarrow} \qquad
\left\{\begin{array}{c} \text{Piecewise non-zero maps} \\ X_s \rightarrow Y_t \rightarrow Z_r	
\end{array}\right\}
\end{equation}

Next, we will show that the composition is associative. Let
\begin{align*}
\Big(\alpha,\{f_{s,t}\}_{s \in S}^{t \in \alpha(s)}\Big) &\in \Hom_{\dCCC}\left(\;\bigoplus_{s \in S} X_s,\bigoplus_{t \in T} Y_t\right) , \\
\Big(\beta,\{g_{t,r}\}_{t \in T}^{r \in \beta(t)}\Big) &\in \Hom_{\dCCC}\left(\;\bigoplus_{t \in T} Y_t,\bigoplus_{r \in R} Z_r\right) , \\
\Big(\gamma,\{h_{r,d}\}_{r \in R}^{d \in \gamma(t)}\Big) &\in \Hom_{\dCCC}\left(\;\bigoplus_{r \in R} Z_r,\bigoplus_{d \in D} A_d\right) ,
\end{align*}
be three morphisms where any defined component morphism is non-zero.  Associativity of the composition holds if and only if for any fixed $s \in S$ and $d \in D$, 
\begin{equation} \label{compositionassoc}
\sum_{{\substack{r \in \beta\big(\alpha(s)\big) \\ \text{ s.t. } d \in \gamma(r)}}} h_{r,d} \circ \left( \sum_{{\substack{t \in \alpha(s) \\ \text{ s.t. } r \in \beta(t)}}} g_{t,r} \circ f_{s,t} \right) \;\;\; = \sum_{{\substack{t \in \alpha(t) \\ \text{ s.t. } d \in \gamma\big(\beta(t)\big)}}} \left( \sum_{{\substack{r \in \beta(t) \\ \text{ s.t. } d \in \gamma(r)}}} h_{r,d} \circ g_{t,r} \right) \circ f_{s,t} \; .
\end{equation}
Using~\eqref{piecewisenonzero} and the additive $\mathbb{F}$-linear structure of $\CCC$, we see that the sets of terms on either side of the equality~\eqref{compositionassoc} are both in bijection with the set of all piecewise non-zero maps $X_s \rightarrow Y_t \rightarrow Z_r \rightarrow A_d$. It follows that the set of terms on either side of~\eqref{compositionassoc} must coincide and as a result, the composition is associative. It is also easy to see that the identity $(\Id_{f.s.(S)},\{\Id_{X_s}\}^{t \in \alpha(s)}_{s \in S})$ preserves morphisms.
\end{proof}
\end{proposition}

The next definitions and proposition focus on transferring the additive and $\mathbb{F}$-linear structure of $\CCC$ to $\dCCC$. This structure will be fundamental to all our further uses of $\dCCC$.


\begin{definition} \label{addition} Define an addition on $\Hom_{\dCCC}\big(\bigoplus_{s \in S} X_s,\bigoplus_{t \in T} Y_t\big)$ as follows:
\begin{equation} \label{additionsumcomp}
\Big(\alpha_1,\{f^1_{s,t}\}_{s \in S}^{t \in \alpha_1(s)}\Big) + \Big(\alpha_2,\{f^2_{s,t}\}_{s \in S}^{t \in \alpha_2(s)}\Big) = \Bigg(\alpha_1 \cup \alpha_2,\{\sigma_{s,t}\}_{s \in S}^{t \in \big(\alpha_1\cup\alpha_2\big)(s)}\Bigg), 
\end{equation}
where
\begin{itemize}
\item for any finite subset $A \subseteq S$, define $\big(\alpha_1 \cup \alpha_2\big)(A) = \alpha_1(A) \cup \alpha_2(A)$ ;
\item for any $s \in S$ and $t \in \big(\alpha_1 \cup \alpha_2\big)(s)$, define $\sigma_{s,t} = 
\left\{\begin{array}{cl}
f^1_{s,t} & \text{if } t \in \alpha_1(s) \backslash \alpha_2(s) \\
f^1_{s,t} + f^2_{s,t} & \text{if } t \in \alpha_1(s) \cap \alpha_2(s) \\
f^2_{s,t} & \text{if } t \in \alpha_2(s) \backslash \alpha_1(s) \\
\end{array}\right.$.
\end{itemize}

Given any $\lambda \in \mathbb{F}$, define a scalar multiplication by $\lambda$ on $\Hom_{\dCCC}\big(\bigoplus_{s \in S} X_s,\bigoplus_{t \in T} Y_t\big)$ as follows:
\begin{equation} \label{scalarmulteq}
\lambda \cdot \Big(\alpha,\{f_{s,t}\}_{s \in S}^{t \in \alpha(s)}\Big) = \Big(\alpha,\{\lambda f_{s,t}\}_{s \in S}^{t \in \alpha(s)}\Big)
\end{equation}

\end{definition}

\begin{definition} \label{zeroob} Define a zero object $0_{\dCCC} = \bigoplus_{0 \in \{0\}} 0_0$ where $0_0 = 0 \in \ob(\CCC)$. Also, define the zero morphism in $\Hom_{\dCCC}\big(\bigoplus_{s \in S} X_s,\bigoplus_{t \in T} Y_t\big)$ as $\big(\Omega, \emptyset\big)$ where $\Omega(A) = \emptyset \subseteq T$ for every finite subset $A \subseteq S$. In particular, this gives

\begin{equation} \label{zeroobjecteq}
\Hom_{\dCCC}(0_{\dCCC},0_{\dCCC}) = \left\{\begin{array}{c} \text{The equivalence} \\
\text{class of } \big(\Omega,\emptyset\big).
\end{array}\right\} \; .
\end{equation}
\end{definition}

\begin{definition} \label{finiteds} Define finite direct sums in $\dCCC$ as follows. Given a pair of objects $\bigoplus_{s \in S} X_s$ and $\bigoplus_{t \in T} Y_t$ of $\dCCC$, we have an object $\bigoplus_{a \in S \sqcup T} A_a$ where
\begin{align*}
A_a = \left\{\begin{array}{cl}X_a & \text{if } a \in S \\ Y_t & \text{if } a \in T\end{array}\right. .
\end{align*}
with projection and inclusion morphisms $p_S,p_T,i_S,i_T$ satisfying 
\begin{align} \label{dsequalities}
p_S \circ i_S = \Id_{\bigoplus_{s \in S}X_s}, && p_T \circ i_T = \Id_{\bigoplus_{t \in T}Y_t}, && i_S \circ p_S \; + \; i_T \circ p_T = \Id_{\bigoplus_{a \in S \sqcup T} A_a}.
\end{align}

Concretely, define $p_S,p_T,i_S,i_T$ as follows:
\begin{align*}
p_S &= \Big(\pi_S,\{\Id_{X_a}\}_{a \in S \sqcup T}\Big) \in \Hom_{\dCCC}\left(\;\bigoplus_{a \in S \sqcup T} A_a,\bigoplus_{s \in S} X_s\right), \\
p_T &= \Big(\pi_T,\{\Id_{Y_a}\}_{a \in S \sqcup T}\Big) \in \Hom_{\dCCC}\left(\;\bigoplus_{a \in S \sqcup T} A_a,\bigoplus_{t \in T} Y_t\right), \\
i_S &= \Big(\iota_S,\{\Id_{X_s}\}_{s \in S}\Big) \in \Hom_{\dCCC}\left(\;\bigoplus_{s \in S} X_s,\bigoplus_{a \in S \sqcup T} A_a\right), \\
i_T &= \Big(\iota_T,\{\Id_{Y_t}\}_{t \in T}\Big) \in \Hom_{\dCCC}\left(\;\bigoplus_{s \in S} X_s,\bigoplus_{a \in S \sqcup T} A_a\right)
\end{align*}
where 
\begin{itemize}
\item $\pi_S(A) = \{a \in A \; | \; a \in S\} \subseteq S$ for any finite subset $A \in S \sqcup T$;
\item $\pi_T(A) = \{a \in A \; | \; a \in T\} \subseteq T$ for any finite subset $A \in S \sqcup T$;
\item $\iota_S(B) = B \subseteq S \sqcup T$ for any finite subset $B \subseteq S$;
\item $\iota_T(C) = C \subseteq S \sqcup T$ for any finite subset $C \subseteq T$.
\end{itemize}
\end{definition}

\begin{proposition} The elements of Definitions~\ref{addition}, \ref{zeroob}, and~\ref{finiteds} make $\dCCC$ an additive $\mathbb{F}$-linear category. In particular:
\begin{itemize}
\item the addition~\eqref{additionsumcomp} and scalar multiplication~\eqref{scalarmulteq} are compatible with the equivalence relation $\sim$ in both arguments;
\item the scalar multiplication is distributive with respect to the addition;
\item the $\Hom$-spaces of $\dCCC$ form $\mathbb{F}$-vector spaces with the zero morphisms $(\Omega,\emptyset)$ and where the inverse of a morphism $(\alpha,\{f_{s,t}\})$ is just $(\alpha,\{-f_{s,t}\})$;
\item the composition of morphisms in $\dCCC$ is $\mathbb{F}$-bilinear;
\item the equalities ~\eqref{zeroobjecteq} and ~\eqref{dsequalities} indeed hold.
\end{itemize}

\begin{proof} First, let's prove that the addition and scalar multiplication are compatible with the equivalence relation $\sim$ that defines morphisms. Let $\lambda \in \mathbb{F}$ and let 
\begin{align}
\Big(\alpha,\{f_{s,t}\}_{s \in S}^{t \in \alpha(s)}\Big) &\sim \Big(\tilde{\alpha}, \{\tilde{f}_{s,t}\}_{s \in S}^{t \in \tilde{\alpha}(s)}\Big) \in \Hom_{\dCCC}\big(\bigoplus_{s \in S} X_s,\bigoplus_{t \in T} Y_t\big) \; , \label{firstequivadd} \\
\Big(\beta,\{g_{s,t}\}_{s \in S}^{t \in \beta(s)}\Big) &\sim \Big(\tilde{\beta}, \{\tilde{g}_{s,t}\}_{s \in S}^{t \in \tilde{\beta}(s)}\Big) \in \Hom_{\dCCC}\big(\bigoplus_{s \in S} X_s,\bigoplus_{t \in T} Y_t\big) \; . \label{secondequivadd}
\end{align}
We have to show that 
\begin{equation}	\label{addsim}
\Big(\alpha,\{f_{s,t}\}_{s \in S}^{t \in \alpha(s)}\Big) + \lambda \cdot \Big(\beta,\{g_{s,t}\}_{s \in S}^{t \in \beta(s)}\Big) \; \sim \; \Big(\tilde{\alpha}, \{\tilde{f}_{s,t}\}_{s \in S}^{t \in \tilde{\alpha}(s)}\Big) + \Big(\tilde{\beta}, \{\lambda \tilde{g}_{s,t}\}_{s \in S}^{t \in \tilde{\beta}(s)}\Big) . 
\end{equation} 
Fix $s \in S$. Then there are three different cases of $t$ to distinguish in order to prove~\eqref{addsim}: 
\begin{enumerate}
\item[(1)] $t \in \big(\alpha(s) \cup \beta(s)\big) \cap \big(\tilde{\alpha}(s) \cup \tilde{\beta}(s)\big)$. There are nine subcases here. To begin with, let $t \in \big(\alpha(s) \cap \beta(s)\big) \cap \big(\tilde{\alpha}(s) \cap \tilde{\beta}(s)\big)$. Then by the equivalences above and the definition of the addition,  the equality $f_{s,t} + \lambda \cdot g_{s,t} = \tilde{f}_{s,t} + \lambda \tilde{g}_{s,t} \in \Hom_\CCC(X_s,Y_t)$ holds (the functions are equal to their $\tilde{}$ equivalent for such $t$). Next, let $t \in \big(\alpha(s) \cap \beta(s)\big) \cap \big(\tilde{\alpha}(s) \backslash \tilde{\beta}(s)\big)$. For such $t$, \eqref{addsim} is satisfied if and only if $f_{s,t} + \lambda \cdot g_{s,t} = \widetilde{f}_{s,t}$, but it holds since the equivalences imply $f_{s,t} = \tilde{f}_{s,t}$ and $g_{s,t} = 0_{s,t}$. The seven other cases are treated similarly.   
\item[(2)] $t \in \big(\alpha(s) \cup \beta(s)\big) \backslash \big(\tilde{\alpha}(s) \cup \tilde{\beta}(s)\big)$. There are three subcases here. For all such $t$, both $\tilde{f}_{s,t}$ and $\tilde{g}_{s,t}$ are not defined. Therefore, the equivalences~\eqref{firstequivadd} and~\eqref{secondequivadd} imply that  any $f_{s,t}$ and $g_{s,t}$ that are defined must be $0_{s,t}$. Checking the requirements of~\eqref{addsim} in all cases comes down to checking that $f_{s,t} + \lambda \cdot g_{s,t} = 0_{s,t} \in \Hom_\CCC(X_s,Y_t)$. For instance, if $t \in \big(\alpha(s) \cap \beta(s)\big) \backslash \big(\tilde{\alpha}(s) \cup \tilde{\beta}(s)\big)$, then only $f_{s,t}$ and $g_{s,t}$ are defined and are zero by ~\eqref{firstequivadd} and~\eqref{secondequivadd}. The requirement of~\eqref{addsim} for this choice of $t$ is that $f_{s,t} + \lambda \cdot g_{s,t} = 0_{s,t}$, but this is obvious by the explanation of the previous sentence.
\item[(3)] $t \in \big(\tilde{\alpha}(s) \cap \tilde{\beta}(s)\big) \backslash \big(\alpha(s) \cup \beta(s)\big)$ is analogous to (2).
\end{enumerate}


Setting $\lambda = 1$ in \eqref{addsim} shows that addition is compatible with $\sim$. To show that scalar multiplication is compatible with $\sim$, we first show that $(\Omega,\emptyset)$ (see Definition \ref{zeroob}) is a neutral element for addition in $\Hom_{\dCCC}\big(\bigoplus_{s \in S} X_s,\bigoplus_{t \in T} Y_t\big)$. Given a morphism $\Big(\alpha,\{f_{s,t}\}_{s \in S}^{t \in \alpha(s)}\Big)$, one has the equivalence $(\Omega,\emptyset) \sim (\alpha,\{0_{s,t}\}_{s \in S}^{t \in \alpha(t)})$ from which neutrality follows. Compatibility of scalar multiplication with $\sim$ follows from $\eqref{addsim}$.\\




Next, we have to show that an arbitrary $\Hom_{\dCCC}\big(\bigoplus_{s \in S} X_s,\bigoplus_{t \in T} Y_t\big)$ are $\mathbb{F}$-vector spaces. We already have a neutral element $(\Omega,\emptyset)$. To explain associativity of addition~\eqref{addition}, consider three morphisms $\Big(\alpha,\{f_{s,t}\}_{s \in S}^{t \in \alpha(s)}\Big)$, $\Big(\beta,\{g_{s,t}\}_{s \in S}^{t \in \alpha(s)}\Big)$, $\Big(\gamma,\{h_{s,t}\}_{s \in S}^{t \in \alpha(s)}\Big)$ to add up. Firstly, note that for any $s \in S$, we have $\big(\alpha(s) \cup \beta(s) \big) \cup \gamma(s) = \alpha(s) \cup \big( \beta(s) \cup \gamma(s)\big)$. Secondly, as addition is compatible with $\sim$, we can assume that for any choice of $s \in S$ and $t \in \alpha(s) \cup \beta(s) \cup \gamma(s)$, $f_{s,t}$, $g_{s,t}$ and $h_{s,t}$ are all defined by letting them be $0_{s,t}$ if they were not already defined. The associativity requirement of~\eqref{addition} then becomes $(f_{s,t} + g_{s,t}) + h_{s,t} = f_{s,t} + (g_{s,t} + h_{s,t})$ for all $s \in S$ and $t \in \alpha(s) \cup \beta(s) \cup \gamma(s)$. The $\Hom_\CCC(X_s,Y_t)$ are already $\mathbb{F}$-vector spaces so the addition~\eqref{addition} is indeed associative. Explaining commutativity of addition and distributivity of scalar multiplication can be done analogously by following these steps:
\begin{enumerate}
\item Using the compatibility of addition and scalar multiplication with $\sim$, assume that all the involved morphisms of $\dCCC$ (a finite number in each case) have the same set map by defining zero component morphisms where needed;
\item Recognise the target property component-wise in the $\Hom$-spaces of the additive $\mathbb{F}$-linear category $\CCC$ and conclude that the property holds in $\dCCC$ as well.
\end{enumerate}
The additive inverse of a morphism in $\dCCC$ is obtained by multiplying it by the scalar $-1 \in \mathbb{F}$ and so the $\Hom$-spaces of $\dCCC$ are indeed $\mathbb{F}$-vector spaces. \\

Next, we must show that composition in $\dCCC$ is $\mathbb{F}$-bilinear. Fix $\lambda \in \mathbb{F}$ as well as morphisms
\begin{align}
\Big(\alpha_1,\{f^1_{s,t}\}_{s \in S}^{t \in \alpha_1(s)}\Big), \Big(\alpha_2,\{f^2_{s,t}\}_{s \in S}^{t \in \alpha_2(s)}\Big) &\in \Hom_{\dCCC}\left(\;\bigoplus_{s \in S} X_s,\bigoplus_{t \in T} Y_t\right) , \label{bilincompositionslot1} \\
\Big(\beta_1,\{g^1_{t,r}\}_{t \in T}^{r \in \beta_1(t)}\Big), \Big(\beta_2,\{g^2_{t,r}\}_{t \in T}^{r \in \beta_2(t)}\Big) &\in \Hom_{\dCCC}\left(\;\bigoplus_{t \in T} Y_t,\bigoplus_{r \in R} Z_r\right) . \label{bilincompositionslot2}
\end{align}
Let's show it for the second argument. What has to be shown is that 
\begin{align*}
\Big(\beta_1,\{g^1_{t,r}\}_{t \in T}^{r \in \beta_1(t)}\Big) \circ \bigg( \Big(\alpha_1,\{f^1_{s,t}\}_{s \in S}^{t \in \alpha_1(s)}\Big) + \lambda \cdot \Big(\alpha_2,\{f^2_{s,t}\}_{s \in S}^{t \in \alpha_2(s)}\Big) \bigg) \\
\sim \bigg( \Big(\beta_1,\{g^1_{t,r}\}_{t \in T}^{r \in \beta(t)}\Big) \circ \Big(\alpha_1,\{f^1_{s,t}\}_{s \in S}^{t \in \alpha_1(s)}\Big) \bigg) + \lambda \cdot \bigg( \Big(\beta_1,\{g^1_{t,r}\}_{t \in T}^{r \in \beta_1(t)}\Big) \circ \Big(\alpha_2,\{f^2_{s,t}\}_{s \in S}^{t \in \alpha_2(s)}\Big) \bigg)
\end{align*}
Since composition, addition and scalar multiplication are all compatible with $\sim$, assume that $\alpha_1 = \alpha_2 = \alpha$ and $\beta_1 = \beta_2 = \beta$ by defining zero component morphisms where needed. Then, we have to show that
\begin{equation*}
\sum_{{\substack{t \in \alpha(s) \\ \text{ s.t. } r \in \beta(t)}}} g^1_{t,r} \circ (f^1_{s,t} + \lambda f^2_{s,t}) =
\left( \sum_{{\substack{t \in \alpha(s) \\ \text{ s.t. } r \in \beta(t)}}} g^1_{t,r} \circ f^1_{s,t} \right) + \lambda \cdot \left( \sum_{{\substack{t \in \alpha(s) \\ \text{ s.t. } r \in \beta(t)}}} g^1_{t,r} \circ f^2_{s,t} \right) \; .
\end{equation*}
This is true since compositions in $\CCC$ are $\mathbb{F}$-bilinear. Linearity in the first argument can be proven in the same way. \\

Next, the category $\dCCC$ must have a zero object. We take it to be $0_{\dCCC} = \bigoplus_{0 \in \{0\}} 0_0$ as of Definition \ref{zeroob}. Since $0_0 = 0 \in \ob(\CCC)$, the only possible component endomorphism of $0_{\dCCC}$ is $0_{0,0} \in \Hom_\CCC(0,0)$ and we directly get equation~\eqref{zeroobjecteq}. \\

Finally, we must show that Definition~\ref{finiteds} indeed define direct sums in $\dCCC$. The only things to show here are the equalities of line~\eqref{dsequalities}. Let's treat each of them separately:
\begin{itemize}
\item[(1)] $p_S \circ i_S = \Id_{\bigoplus_{s \in S}X_s}$. The set map of this composition is $\pi_S \circ \iota_S$. As the map $\iota_S$ embeds a finite subset of $S$ into the disjoint union $S \sqcup T$ and $\pi_S$ sends a finite subset of this disjoint union to the collection of its elements belonging to $S$, we get $\pi_S \circ \iota_S = \Id_{\text{f.s.}(S)}$. For any fixed $s \in S$, as $\iota_S(s) = \{s\} = \big(\pi_S \circ \iota_S\big)(s)$, there can be only one component morphism with component $X_s$ in the composition $p_S \circ i_S $ and it has to have codomain $X_s$ as well. Moreover, this component morphism is given by $\Id_{X_s} \circ \Id_{X_s} = \Id_{X_s}$ according to the composition rule in $\dCCC$;
\item[(2)] $p_T \circ i_T = \Id_{\bigoplus_{t \in T}Y_t}$ can be proven exactly as in (1) but with $T$ playing the role of $S$;
\item[(3)] $i_S \circ p_S \; + \; i_T \circ p_T = \Id_{\bigoplus_{a \in S \sqcup T} A_a}$. By definition, the set map of the left side is $(\iota_S \circ p_S) \cup (\iota_T \cup p_T)$. Let $E \subseteq S \sqcup T$ be a finite set. Then $\big(\iota_S \circ p_S\big)(E) \subseteq S \sqcup T$ is the collection of elements of $E$ that belong to $S$ and similarly for $T$ so that $\big((\iota_S \circ p_S) \cup (\iota_T \cup p_T)\big)(E) = E$. For the component morphisms, fix $x \in S \sqcup T$ and observe that:
\begin{align*}
\big(\iota_S \circ p_S\big)(x) = \left\{ \begin{array}{cl}
\{x\} & \text{if } x \in S \\
\emptyset & \text{if } x \in T
\end{array} \right. \; ,  && \big(\iota_T \circ p_T\big)(x) = \left\{ \begin{array}{cl}
\emptyset & \text{if } x \in S \\
\{x\} & \text{if } x \in T
\end{array} \right. \; .
\end{align*}
It follows that for $s \in S \subseteq S \sqcup T$, the only possible component morphism with domain $X_s$ has to have codomain $X_s$ and similarly for $t \in T \subseteq S \sqcup T$. For an arbitrary $x \in S$, the definition of the addition of morphisms make the component morphisms of $i_S \circ p_S \; + \; i_T \circ p_T$ correspond to those of $i_S \circ p_S$ if $x \in S$ and to those of $i_T \circ p_T$ if $x \in T$. In both cases, we obtain the identity of the corresponding object of $\CCC$. The equality $i_S \circ p_S \; + \; i_T \circ p_T = \Id_{\bigoplus_{a \in S \sqcup T} A_a}$ is then assured.
\end{itemize}

This proves the equalities of line~\eqref{dsequalities} showing that $\dCCC$ has finite direct sums. In conclusion, $\dCCC$ is an additive $\mathbb{F}$-linear category, just like $\CCC$.
\end{proof}
\end{proposition}

\begin{definition} \label{arbitrarycoprods} Let $\Big(\bigoplus_{s \in S_i} X_s^i\Big)_{i \in I}$ be a family of objects in $\dCCC$. Define their coproduct to be the object $\bigoplus_{a \in \bigsqcup_{i \in I} S_i} A_a$ where $A_a = X_a^{i_0}$  for $a \in S_{i_0} \subseteq \bigsqcup_{i \in I} S_i$. The structural injections are given by 

$$ \inj_{S_{i_0}} = \Big(\iota_{S_{i_0}},\{\Id_{X_s^{i_0}}\}_{s \in S_{i_0}}\Big) \in \Hom_{\dCCC}\left(\;\bigoplus_{s \in S_{i_0}} X_s^{i_0},\bigoplus_{a \in \bigsqcup_{i \in I} S_i} A_a\right) \\
 $$
where $\iota_{S_{i_0}}(B) = B \subseteq \bigsqcup_{i \in I} S_i$ for any finite subset $B \subseteq S_{i_0}$.
\end{definition}

\begin{proposition} \label{coproductsproof} The elements of Definition~\ref{arbitrarycoprods} indeed define arbitrary coproducts in $\dCCC$.

\begin{proof} We have to show that the map
\begin{eqnarray} \label{mapcoproducts}
M : \; \Hom_{\dCCC}\left(\;\bigoplus_{a \in \bigsqcup_{i \in I} S_i} A_a,\bigoplus_{r \in R} Z_r\right) & \longrightarrow & \prod_{i \in I} \; \Hom_{\dCCC}\left(\bigoplus_{s \in S_i} X_s^i,\bigoplus_{r \in R}, Z_r\right) \; ,\\
F & \longmapsto & \big(F \circ \inj_{S_i}\big)_{i \in I} \notag
\end{eqnarray}
is bijective and functorial in $\bigoplus_{r \in R}, Z_r$. Let $\left(\alpha,\{f_{a,r}\}_{a \in \sqcup_{i \in I} S_i}^{r \in \alpha(a)}\right)$ be a morphism in the domain of $M$ and fix $i_0 \in I$. Then it is straightforward to see that 
\begin{equation} \label{directsumrequirement}
\left(\alpha,\{f_{a,r}\}_{a \in \sqcup_{i \in I} S_i}^{r \in \alpha(a)}\right) \circ \inj_{S_{i_0}} = \left(\alpha|_{\text{f.s.}( S_{i_0})},\{f_{a,r}\}_{a \in S_{i_0}}^{r \in \alpha(a)}\right).
\end{equation}
Since the union $\sqcup_{i \in I} S_i$ is disjoint, the collection of maps $(\alpha|_{\text{f.s.}( S_{i})})_{i \in I}$ uniquely determines $\alpha$. Also, one has $ \{f_{a,r}\}_{a \in \sqcup_{i \in I} S_i}^{r \in \alpha(a)} = \prod_{i \in I}  \{f_{s,r}\}_{s \in S_i}^{r \in \alpha(s)}$. Without loss of generality, all component morphisms $f_{a,r}$ that are defined are non-zero. For any fixed $i_0$, the right hand side morphism of line~\eqref{directsumrequirement} is reduced in the same sense. We conclude that the morphism $\left(\alpha,\{f_{a,r}\}_{a \in \sqcup_{i \in I} S_i}^{r \in \alpha(a)}\right)$ uniquely determines the collection of morphisms $\left(\alpha|_{\text{f.s.}( S_{i})},\{f_{a,r}\}_{a \in S_{i}}^{r \in \alpha(a)}\right)_{i \in I}$, hence $M$ is one to one. Conversely, any collection of morphisms in the codomain of $M$ can be combined into a morphism of $\Hom_{\dCCC}\left(\;\bigoplus_{a \in \bigsqcup_{i \in I} S_i} A_a,\bigoplus_{r \in R} Z_r\right)$ because the union $\bigsqcup_{i \in I} S_i$ is disjoint. This shows that $M$ is also surjective. \\

For functoriality of $M$, let $\left(\gamma,\{g_{r,u}\}_{r \in R}^{u \in \gamma(r)}\right) \in \Hom_{\dCCC}\left(\;\bigoplus_{r \in R} Z_r, \bigoplus_{u \in U} C_u\right)$. Thanks to line~\eqref{directsumrequirement}, we can write that $M\left(\alpha,\{f_{a,r}\}_{a \in \sqcup_{i \in I} S_i}^{r \in \alpha(a)}\right)$ composed with $\left(\gamma,\{g_{r,u}\}_{r \in R}^{u \in \gamma(r)}\right)$ gives
\begin{equation} \label{functorialitycoproduct}
\left(\gamma \circ (\alpha|_{\text{f.s.}( S_{i_0})}),\left\{\sum_{{\substack{r \in \big(\alpha|_{\text{f.s.} (S_{i_0})}\big)(s_{i_0}) \\ \text{ s.t. } u \in \gamma(r)}}} h_{r,u} \circ f_{s,r}\right\}_{s_{i_0} \in S_{i_0}}^{u \in \gamma \circ (\alpha|_{\text{f.s.}( S_{i_0})})(s_{i_0})} \right) \; .
\end{equation}

However, $\gamma \circ (\alpha|_{\text{f.s.} (S_{i_0})}) = (\gamma \circ \alpha)|_{\text{f.s.}( S_{i_0})}$ and $\big(\alpha|_{\text{f.s.}( S_{i_0})}\big)(s_{i_0}) = \alpha(s_{i_0})$ for any $s_{i_0} \in S_{i_0}$.
Thus, the morphism of line~\eqref{functorialitycoproduct} is equal to the result of the application of $M$ to the composition $\left(\gamma,\{g_{r,u}\}_{r \in R}^{u \in \gamma(r)}\right) \circ \left(\alpha,\{f_{a,r}\}_{a \in \sqcup_{i \in I} S_i}^{r \in \alpha(a)}\right)$.  In conclusion, the map $M$ of line~\eqref{mapcoproducts} is both bijective and functorial in $\bigoplus_{r \in R} Z_r$ as required. This in turn proves that $\dCCC$ is closed under taking arbitrary coproducts.
\end{proof}
\end{proposition}

\begin{definition} \label{embeddingfun} Define an inclusion functor $\mathcal{I} : \CCC \rightarrow \dCCC$ as follows:
\begin{align*}
X & \longmapsto \bigoplus_{0 \in \{0\}} X_0 \quad \text{where } X_0 = X \\
[X \overset{f}{\rightarrow} Y] & \longmapsto \left(\Id_{\{0\}},\{f_{0,0} = f\}_{0 \in \{0\}}^{0 \in \{0\}}\right)
\end{align*}
\end{definition}

\begin{proposition}
The inclusion functor $\mathcal{I}$ is fully faithful and $\mathbb{F}$-linear. In other words, there are natural $\mathbb{F}$-linear bijections 
$$ \Hom_{\dCCC}\big(\mathcal{I}(X),\mathcal{I}(Y)\big) = \Hom_\CCC(X,Y),$$ 
Moreover, every $\bigoplus_{s \in S} X_s \in \ob(\dCCC)$ is a direct sum (in $\dCCC$) of its terms $\mathcal{I}(X_s) \in \ob(\dCCC)$.
\begin{proof} Let $X$ and $Y$ be objects of $\CCC$. By Definition \ref{embeddingfun}, the sets associated to both objects $\mathcal{I}(X)$ and $\mathcal{I}(Y)$ is $\{0\}$ which has only one element. Without loss of generality, the set map of an arbitrary morphism $\Hom\dCCC\big(\mathcal{I}(X),\mathcal{I}(Y)\big)$ can be taken to send $\{0\}$ to $\{0\}$. Thus, an arbitrary morphism of this set is equivalent to one of the form $\left(\Id_{\{0\}},\{f_{0,0} = f\}_{0 \in \{0\}}^{0 \in \{0\}}\right)$ where $f : X = X_0 \rightarrow Y_0 = Y$ is a morphism in $\CCC$. It follows that the mapping 
$$ [X \overset{f}{\rightarrow} Y] \quad \longmapsto \quad \left(\Id_{\{0\}},\{f_{0,0} = f\}_{0 \in \{0\}}^{0 \in \{0\}}\right) \; ,$$
is bijective and $\mathbb{F}$-linear. It is obvious that it also preserves compositions, so $\mathcal{I}$ is a fully faithful $\mathbb{F}$-linear functor. Finally, an arbitrary object $\bigoplus_{s \in S} X_s \in \ob(\dCCC)$ is a direct sum of the objects $\{\mathcal{I}(X_s)\}_{s \in S}$ by Proposition~\ref{coproductsproof}.
\end{proof}
\end{proposition}

\subsection{Monoidal structure on $\dCCC$}
In this subsection, let $\CCC$ denote a $\mathbb{F}$-linear monoidal category with tensor product $\otimes$, associativity isomorphisms $\{a_{X,Y,Z}\}_{X,Y,Z \in \ob(\CCC)}$ and unit $(\mathds{1},l,r)$. 
The goal of this subsection is to define a natural monoidal structure on $\dCCC$.
%
%
%
%
%


\begin{definition} \label{tensorsum} Define a tensor product on $\otimes_{\dCCC} : \; \dCCC \times \dCCC \rightarrow \dCCC$ as follows:
\begin{itemize}
\item it sends a pair of objects $\left(\bigoplus_{s \in S} X_s \bigoplus_{t \in T} Y_t\right)$ to the object $\bigoplus_{(s,t) \in S \times T} (X_s \otimes_\CCC Y_t) ; $
\item it sends a pair of morphisms $\Big(\alpha,\{f_{s,\tilde{s}}\}_{s \in S}^{\tilde{s} \in \alpha(s)}\Big),\Big(\beta,\{g_{t,\tilde{t}}\}_{t \in T}^{\tilde{t} \in \beta(t)}\Big)$ to the morphism
\begin{align} \label{tensormorphismsmap}
&\alpha \otimes \beta \; : \big\{(s_i,t_i)\big\}_{i=1}^n \longmapsto \bigcup_{i=1}^{n} \big(\alpha(s_i) \times \beta(t_i)\big) \\
&\Big\{f_{s,\tilde{s}} \otimes g_{t,\tilde{t}} \Big\}_{(s,t) \in S \times T}^{(\tilde{s},\tilde{t}) \in \alpha(s) \times\beta(t) = \big(\alpha \otimes \beta\big)(s,t)}
\end{align}
\end{itemize}
Note that the rule $\emptyset \times A = \emptyset$ for any set $A$ is assumed in the above. 
\end{definition}

\begin{definition} \label{assosum} Let $\bigoplus_{s \in S} X_s, \bigoplus_{t \in T} Y_t, \bigoplus_{r \in R} Z_r \in \ob(\dCCC)$. Define associativity morphisms for the tensor product $\otimes_{\dCCC}$ as follows:
\begin{equation*}
a_{(\oplus_S X_s, \oplus_T Y_t, \oplus_R Z_r)}^{\dCCC} = \Bigg(\alpha : \Big\{\big((s_i,t_i),r_i)\big)\Big\}_{i=1}^n \mapsto \Big\{\big(s_i,(t_i,r_i)\big)\Big\}_{i=1}^n,\Big\{a_{X_s,Y_t,Z_r}\Big\}_{\big((s,t),r)\big) \in (S \times T) \times R}\Bigg)
\end{equation*}
\end{definition}

\begin{definition} \label{unitsum} Define a unit object $\mathds{1}_{\dCCC} = \mathcal{I}(\mathds{1}) = \bigoplus_{0 \in \{0\}} \mathds{1}_0$ where $\mathds{1}_0 = \mathds{1} \in \ob(\CCC)$ and a left unit $l^{\C_{\oplus}}_{-}$ by
$$ l^{\dCCC}_{\bigoplus_{s \in S} X_s} = \Big(\alpha : \{(0,s_i)\}_{i=1}^n \mapsto \{s_i\}_{i=1}^n,\big\{l_{X_s}\big\}_{(0,s) \in \{0\}\times S}^{s \in \alpha(0,s)}\Big) \in \Hom_{\dCCC}\left(\mathds{1}_{\dCCC} \otimes \bigoplus_{s \in S} X_s,\bigoplus_{s \in S} X_s\; \right). $$
Right units are defined similarly.
\end{definition}

\begin{proposition} The elements of Definitions~\ref{tensorsum},~\ref{assosum},~\ref{unitsum} define a monoidal structure on $\dCCC$. In particular:
\begin{itemize}
\item $\otimes_{\dCCC}$ is a bifunctor $\dCCC \times \dCCC \rightarrow \dCCC$;
\item $(\mathds{1}_{\dCCC},l^{\dCCC}_{-},r^{\dCCC}_{-})$ is a unit for this tensor product;
\item The isomorphisms $a_{-,-,-}^{\dCCC}$ are well defined and trinatural;
\item The pentagon and triangle axioms are satisfied.
\end{itemize}
\begin{proof} 
Let's first show that $\otimes_{\dCCC}$ is a bifunctor. Note that the set map~\eqref{tensormorphismsmap} commutes with unions since both $\alpha$ and $\beta$ do. A less obvious fact is that the effect of $\otimes_{\dCCC}$ on morphisms is compatible with $\sim$. 
Fix $\bigoplus_{\ell \in L} D_\ell \in \ob(\dCCC)$ and consider the operation 
\begin{equation} \label{tensorproductslot}
\left( \bigoplus_{\ell \in L} D_\ell \right) \otimes_{\dCCC} - \; .
\end{equation}

Let $\Big(\beta,\{g_{s,t}\}_{s \in S}^{t \in \beta(s)}\Big) \sim \Big(\tilde{\beta}, \{\tilde{g}_{s,t}\}_{s \in S}^{t \in \tilde{\beta}(s)}\Big) \in \Hom_{\dCCC}\big(\bigoplus_{s \in S} X_s,\bigoplus_{t \in T} Y_t\big)$ and consider their respective images under~\eqref{tensorproductslot}:
\begin{align}
\bigg(\Id_{\text{f.s.}(L)} \otimes \beta \; : \big\{(\ell_i,s_i)\big\}_{i=1}^n \mapsto \bigcup_{i=1}^{n} \big(\{\ell_i\} \times \beta(s_i)\big) , \Big\{ \Id_{X_\ell} \otimes g_{s,t} \Big\}_{(\ell,s) \in L \times S}^{(\ell,t) \in \{\ell\} \times \beta(s)}\bigg) , \label{tensormorphism1} \\
\bigg(\Id_{\text{f.s.}(L)} \otimes \tilde{\beta} \; : \big\{(\ell_i,s_i)\big\}_{i=1}^n \mapsto \bigcup_{i=1}^{n} \big(\{\ell_i\} \times \tilde{\beta}(s_i)\big) , \Big\{ \Id_{X_\ell} \otimes \tilde{g}_{s,t} \Big\}_{(\ell,s) \in L \times S}^{(\ell,t) \in \{\ell\} \times \tilde{\beta}(s)}\bigg) \; . \label{tensormorphism2}
\end{align}
Let $(\ell,s)\in L \times S$. Then
\begin{align*}
(x,y) \in \big(\Id_{\text{f.s.}(L)} \otimes \beta\big) \big((\ell,s)\big) \backslash \big(\Id_{\text{f.s.}(L)} \otimes \tilde{\beta}\big)\big((\ell,s)\big) \quad &\Leftrightarrow \quad x = \ell \text{ and } y \in \beta(s) \backslash \; \tilde{\beta}(s) \; ; \\
(x,y) \in \big(\Id_{\text{f.s.}(L)} \otimes \beta\big)\big((\ell,s)\big) \cap \big(\Id_{\text{f.s.}(L)} \otimes \tilde{\beta}\big)\big((\ell,s)\big) \quad &\Leftrightarrow \quad x = \ell \text{ and } y \in \beta(s) \cap \; \tilde{\beta}(s) \; ; \\
(x,y) \in \big(\Id_{\text{f.s.}(L)} \otimes \tilde{\beta}\big)\big((\ell,s)\big) \backslash \big(\Id_{\text{f.s.}(L)} \otimes \beta\big)\big((\ell,s)\big) \quad &\Leftrightarrow \quad x = \ell \text{ and } y \in \tilde{\beta}(s) \backslash \; \beta(s) \; .
\end{align*}
It follows immediately that \eqref{tensormorphism1} $\sim$ \eqref{tensormorphism2}. Additionally, we see that \eqref{tensorproductslot} sends identity morphisms to identity morphisms and zero morphisms to zero morphisms. The last thing to check in order to prove that \eqref{tensorproductslot} is a functor is that it preserves compositions in $\dCCC$. Let 
$$ \Big(\beta,\{g_{t,r}\}_{t \in T}^{r \in \beta(t)}\Big) : \bigoplus_{t \in T} Y_t \rightarrow \bigoplus_{r \in R} Z_r \quad \text{ and } \quad \Big(\alpha,\{f_{s,t}\}_{s \in S}^{t \in \alpha(s)}\Big) : \bigoplus_{s \in S} X_s \rightarrow \bigoplus_{t \in T} Y_t \; . $$
Since $\alpha$ and $\beta$ commute with unions, the equality $\Id_{\text{f.s.}(L)} \otimes (\beta \circ \alpha) = (\Id_{\text{f.s.}(L)} \otimes \beta) \circ (\Id_{\text{f.s.}(L)} \otimes \alpha)$ holds.
Let $(\ell,s) \in L \times S$ and $(\ell,r) \in \{\ell\} \times \big(\beta \circ \alpha\big)(s)$. The component morphism $(\ell,s) \mapsto (\ell,r)$ resulting from the application of~\eqref{tensorproductslot} to the composition $\Big(\beta,\{g_{t,r}\}_{t \in T}^{r \in \beta(t)}\Big) \circ \Big(\alpha,\{f_{s,t}\}_{s \in S}^{t \in \alpha(s)}\Big)$ is
\begin{align*}
\sum_{\substack{(\ell,t) \in \{\ell\} \times \alpha(s) \\ \text{ s.t. } (\ell,r) \in \{\ell\} \times \beta(t)}} (\Id_{D_\ell} \otimes g_{t,r}) \circ (\Id_{D_\ell} \otimes f_{s,t}) \; ,
\end{align*}
which is precisely the $(\ell,s) \mapsto (\ell,r)$ component morphism of the composition
$$ \bigg((\Id_{\bigoplus D_\ell}) \otimes_{\dCCC} \Big(\beta,\{g_{t,r}\}_{t \in T}^{r \in \beta(t)}\Big)\bigg) \circ \bigg((\Id_{\bigoplus D_\ell}) \otimes_{\dCCC}\Big(\alpha,\{f_{s,t}\}_{s \in S}^{t \in \alpha(s)}\Big)\bigg) \; . $$
We conclude that~\eqref{tensorproductslot} is indeed a functor and similarly, so is $- \otimes_{\dCCC}  \left( \bigoplus_{\ell \in L} D_\ell \right)$ and ultimately, that $\otimes_{\dCCC}$ is a bifunctor. The triple $(\mathds{1}_{\dCCC},l^{\dCCC}, r^{\dCCC})$ is a unit for this tensor product since the natural bijective map $\{0\} \times A \cong A$ for any set $A$ make the triangle axiom requirement in $\dCCC$ reduces to triangle axiom requirements in $\C$.\\

To prove that $a_{-,-,-}^{\dCCC}$ is an associativity constraint for $\dCCC$, note the obvious map $A \times (B \times C) \cong (A \times B) \times C$ is bijective and natural in every argument, so the pentagon axiom requirement in $\dCCC$ reduces to pentagon axiom requirements in $\C$.
\end{proof}
\end{proposition}
The proof of the following proposition is straightforward.

\begin{proposition} The inclusion functor $\mathcal{I} : \CCC \rightarrow \dCCC$ from Definition \ref{embeddingfun} is monoidal.
\end{proposition}

\subsection{Braiding and twist on $\dCCC$}
In this subsection, let $\CCC$ denote a braided, $\mathbb{F}$-linear and monoidal category with braidings $\{c_{X,Y}\}_{X,Y \in \ob(\CCC)}$ and possibly with twists $\{\theta_X\}_{X \in \ob(\CCC)}$ satisfying the balancing axiom.

\begin{definition} \label{braidingcompletion} Let $\bigoplus_{s \in S} X_s, \bigoplus_{t \in T} Y_t \in \ob(\dCCC)$. Define braiding isomorphisms in $\dCCC$ as follows:
\begin{equation*}
c_{(\oplus_S X_s, \oplus_T Y_t)}^{\dCCC} = \Bigg(\alpha : \Big\{(s_i,t_i)\Big\}_{i=1}^n \mapsto \Big\{(t_i,s_i)\Big\}_{i=1}^n,\Big\{c_{X_s,Y_t}\Big\}_{(s,t) \in S \times T} \Bigg)
\end{equation*}
\end{definition}

\begin{definition} \label{twistcompletion} If $\CCC$ has twists, let $\bigoplus_{s \in S} X_s \in \ob(\dCCC)$ and define twist isomorphisms in $\dCCC$ as follows:
\begin{equation*}
\theta_{(\bigoplus_{s \in S} X_s)}^{\dCCC} = \Bigg( \Id_{\text{f.s.}(S)},\Big\{\theta_{X_s}\Big\}_{s \in S} \Bigg)
\end{equation*}
\end{definition}

\begin{proposition} The braiding and twist of Definitions \ref{braidingcompletion} and \ref{twistcompletion} makes $\dCCC$ a braided $\mathbb{F}$-linear monoidal category (with twists if $\CCC$ has twists). In particular:
\begin{itemize}
\item the braiding $c^{\dCCC}_{-,-}$ and twist $\theta^{\dCCC}_{-}$ are natural isomorphisms in $\dCCC$ (in every argument);
\item the hexagon and balancing axioms are  satisfied.
\end{itemize}
\begin{proof} Naturality of braiding and twist isomorphisms follow from naturality of the associated set maps and of the component maps they are composed of. The hexagon and balancing axioms requirements in $\dCCC$ reduce to the analogues in $\CCC$ for every component.  
\end{proof}
\end{proposition}

The following Proposition is also straightforward:
\begin{proposition} The inclusion functor $\mathcal{I} : \CCC \rightarrow \dCCC$ from Definition \ref{embeddingfun} is braided monoidal and preserves twist.
\end{proposition}

With this setup, a braided $\mathbb{F}$-linear category $\CCC$ with twists can be replaced by $\dCCC$ where infinite direct sums are needed. We will make use of the framework $\dCCC$ in the following section.

\section{Constructing Lattice VOAs} \label{sec:application}

In this section we construct even lattice VOAs as simple current extensions in a category of modules for the Heisenberg VOA. As the simple current from which we build the lattice VOA has infinite order, we must work in the completion category $\dCCC$ in order to make use of the existing theory for simple current extensions of VOAs.  




%
\subsection{The Heisenberg and even lattice VOAs}

Let $\hat{\mathfrak{h}}$ denote the Heisenberg Lie algebra over $\CC$ with vector space basis $\{\kappa, b_n \; | \; n \in \mathbb{Z}\}$ and Lie bracket
\[ [b_n,b_m]=n\delta_{n+m,0} \, \kappa \qquad \text{ and } \qquad [\kappa,b_n]=0. \] 
Fix the triangular decomposition $\hat{\mathfrak{h}} = \Span_\CC\{b_n \; | \; n < 0\} \oplus (\CC.b_0 \oplus \CC.\kappa) \oplus \Span_\CC\{b_n \; | \; n > 0\}$ with Cartan subalgebra $(\CC.b_0 \oplus \CC.\kappa)$.
\begin{definition} Let $\lambda \in \CC$. Define the \textit{Fock space} $\F_{\lambda}$ to be the free $\hat{\mathfrak{h}}$-module generated by a highest weight vector $| \lambda \rangle$ of highest weight given by $b_0. | \lambda \rangle = \lambda | \lambda \rangle$ and $\kappa.| \lambda \rangle = | \lambda \rangle$. 
\end{definition}
Under the natural identification $\F_\lambda = \mathbb{C}[b_{n}]_{n < 0}.| \lambda \rangle$ as vector spaces, an element $b \in \F_\lambda$ is just a polynomial in the variables $\{b_n \; | \; n < 0\}$. The Lie algebra $\hat{\mathfrak{h}}$ acts on $b$ as follows:
\begin{align*}
b_0.b &= \lambda b ,      &      b_n.b &= b_nb \quad \qquad \; \; \text{ if $n<0$} , \\
\kappa.b &= b ,    &     b_n.b &= n \, \partial_{b_{-n}}(b) \quad \text{ if $n>0$}.
\end{align*}
Recall that $\F_0$ can be given a vertex operator algebra structure \cite[Chapter 2]{BF}:
\begin{definition} The \textit{Heisenberg vertex operator algebra} $\mathcal{H} = (\F_0,|0\rangle,Y,T,\omega)$ is the VOA defined by the following data:
\begin{itemize}
\item a $\mathbb{Z}_+$-gradation deg($b_{j_1} \dotsb b_{j_k})=-\sum_{i=1}^k j_i \, $;
\item a vacuum vector $|0\rangle \in \F_0 \,$;
\item vertex operators $Y(-,z)$ defined inductively by
\begin{equation*}
Y(b_{-1},z) = b(z) =	\sum\limits_{n \in \mathbb{Z}} b_nz^{-n-1} \quad \text{ and } \quad Y(b_{j_1} \dotsb b_{j_k},z) = \frac{:\partial_z^{-j_1-1}b(z) \dotsb \partial_z^{-j_k-1}b(z):}{(-j_1-1)! \dotsb (-j_k-1)!}
\end{equation*}
where $:X(z)Y(z):$ denotes the \textit{normally ordered product} of the fields $X(z)$ and $Y(z)$;
\item a translation operator $T$ defined inductively by $T( | 0 \rangle) = 0$ and $[T,b_i]=-i \, b_{i-1} \,$;
\item a conformal vector $\omega = b_{-1}^2 \in \F_0$ of central charge $1$.
\end{itemize}
\end{definition}


The category of modules for $\mathcal{H}$ on which $b_0$ acts semisimply is semisimple and its simple modules are the Fock spaces $\F_{\lambda}$. Recall that the skeleton of a category is the full subcategory containing precisely one representative of each isomorphism class of objects. 

\begin{definition}\label{skel} Let $\C$ be the skeleton of the full subcategory generated by the Fock spaces with index $\lambda \in \mathbb{R}$. 	
\end{definition}
By \cite[Theorem 2.3]{CKLR}, the category $\C$ is a vertex tensor category in the sense of Huang-Lepowsky \cite{HL}. It is also a rigid braided monoidal category with twist as follows:
\begin{itemize}
\item the tensor product on $\mathcal{C}$ is given by $\F_{\lambda_1} \otimes \F_{\lambda_2} = \F_{\lambda_1+\lambda_2}$;
\item the associativity constraint $a_{\F_{\lambda_1},\F_{\lambda_2},\F_{\lambda_3}}:\F_{\lambda_1+\lambda_2+\lambda_3} \to \F_{\lambda_1+\lambda_2+\lambda_3}$ is given by the identity;
\item the braiding is given by $c_{\F_{\lambda_1},\F_{\lambda_2}}=e^{\pi i \lambda_1 \lambda_2}Id_{\F_{\lambda_1+\lambda_2} }$;
\item the twist is given by $\theta_{\F_{\lambda}}=e^{\pi i \lambda^2}Id_{\F_{\lambda}}$. 
\end{itemize}
\begin{remark} Notice that any $\F_\lambda \in \ob(\C)$ is a simple current since $\F_\lambda \otimes \F_{-\lambda} = \F_0 = \mathcal{H}$. The tensor product of Fock spaces in $\C$ also indicates that the left and right duals of $\F_\lambda$ have to be $\F_{-\lambda}$. The corresponding evaluation and coevaluation morphisms can then be fixed in terms of scalars by Schur's Lemma. 
\end{remark}

Let $N \in \ZZ_{>0}$ and form the even lattice $L=\sqrt{2N}\mathbb{Z}$. By the Reconstruction Theorem, the vector space
\[ V_L:= \bigoplus\limits_{\lambda \in L} \F_{\lambda} \;.\]

can be given the structure of a VOA as outlined in \cite[Proposition 5.2.5]{BF}. By \cite{D}, the simple modules of $V_L$ are of the form
\[ \F_{L + x}:= \bigoplus\limits_{\lambda \in L} \F_{\lambda + x} \; ,\]

where $x \in L^*:= \frac{1}{\sqrt{2N}}\mathbb{Z}$. In fact, the isomorphism class of $\F_{L + x}$ only depends on the coset of $x \in L^*/L$. Consider the following complete set of representatives of isomorphism classes of simple modules for $V_L$: 
\begin{equation} \label{latticesimple}
\left\{ \F_{L + \frac{a}{\sqrt{2N}}} \right\}_{a=0}^{a=2N-1}
\end{equation}
The choice of representatives \eqref{latticesimple} amounts to choosing the section $s$ of the short exact sequence $0 \rightarrow L \rightarrow L^* \rightarrow L^*/L \rightarrow 0$ given by $s(\bar{x}) = \frac{a_x}{\sqrt{2N}}$ where $a_x \in \{0, \ldots , 2N-1\}$ is such that $x \equiv \frac{a_x}{\sqrt{2N}} \bmod L$. The 2-cocycle associated with this choice of section is:
\begin{align} 
k: \; L^*/&L \times L^*/L \longrightarrow L \label{cocyclesimplecurr} \\
&(\bar{x},\bar{y}) \mapsto \left\{ \begin{array}{cl} \sqrt{2N} & if \; s(\bar{x})+s(\bar{y}) \geq \sqrt{2N} \\ 0 & else \end{array}\right. \notag
\end{align}


\begin{proposition} The skeleton of the category of modules of $V_L$ is semisimple with simple objects \eqref{latticesimple}. It is a monoidal braided category with twist as follows:
\begin{itemize} 
\item the tensor product is given by $\F_{L + x} \otimes \F_{L + y} = \F_{L + x+y-k(x,y)}$;
\item the associativity constraint is given by $a_{\F_{L + x},\F_{L + y},\F_{L + z}}=(-1)^{x \cdot k(y,z)} \Id_{\F_{L + x+y+z-k(x,y)-k(x+y,z)}}$;
\item the braiding is given by $c_{\F_{L + x},\F_{L + y}} = e^{ \pi i xy} \Id_{\F_{L + x+y-k(x,y)}}$;
\item the twist is given by $\theta_{\F_{L + x}} = e^{2\pi i x^2} \Id_{\F_{L + x}}$.
\end{itemize}
\end{proposition}

See \cite{D,DL,LL} for details.

\subsection{Constructing $V_L$ from simple currents}

The Fock spaces $\F_{\lambda} \in \ob(\C)$ are simple currents in both $\C$ and $\dCCC$ since $\F_{\lambda} \otimes \F_{-\lambda} = \F_0 = \mathcal{H}$. Now, for any $k \in \mathbb{Z}$, one has $\F_{\lambda}^k=\F_{k\lambda}$, so $\theta_{\F_{\lambda}^k}=e^{\pi i (k\lambda)^2} \Id_{\F_{\lambda}^k}$. Hence, if $\lambda =\sqrt{2N}m \in \sqrt{2N}\mathbb{Z}$ for some $N \in \mathbb{Z}_{>0}$, we have
\[ \theta_{\F_{\lambda}^k} = e^{\pi i k^2 \lambda^2} \Id_{\F_{\lambda}^k} = e^{\pi i k^2 2N m^2} \Id_{\F_{\lambda}^k} = \Id_{\F_{\lambda}^k}. \]

Fix $N \in \ZZ_{>0}$ and set $L = \sqrt{2N} \ZZ$ as in the previous subsection. By \cite[Theorem 3.12]{CKL}, the vector space
\[V_L = \bigoplus\limits_{n \in \mathbb{Z}} \F_{\lambda}^n \cong \bigoplus\limits_{n \in \mathbb{Z}} \F_{\sqrt{2N}n}\]

is a Vertex Operator Algebra. Because $\lambda$ is real, $\mathcal{C}$ has vertex tensor category structure in the sense of Huang-Lepowsky by \cite[Theorem 2.3]{CKLR}, and clearly $\mathcal{H}=\F_0$ is a subalgebra of $V_L$. By \cite[Theorem 3.13]{CKL} $V_L$ is also a simple commutative $\C_{\oplus}$-algebra object. For the remainder of this document, we will interpret $V_L$ as such. \\

Let $\mu$ and $u$ denote the multiplication and unit maps of the algebra object $V_L$. By Schur's Lemma and the fact that $\dCCC$ is skeletal (because $\C$ is skeletal), both maps $\mu$ and $u$ can be efficiently described in terms of collections of scalars as follows:
\begin{align} \label{multiplicationlattice}
\mu \quad &\sim \quad \Big(\big\{(\lambda_1,\lambda_2)\big\} \mapsto \{\lambda_1 + \lambda_2\}, \{\mu_{\lambda_1,\lambda_2}\}_{(\lambda_1,\lambda_2) \in L^2}\Big) \in \Hom_{\dCCC}(V_L \otimes_{\dCCC} V_L,V_L) \\
u \quad &\sim \quad \Big(\{0\} \mapsto \{0\},\{u_0\}_{0 \in \{0\}}\Big) \in \Hom_{\dCCC}(\F_0, V_L) \notag
\end{align}

where the scalars $\mu_{\lambda_1,\lambda_2}$ and $u_0$ correspond to the component maps $\mu|_{\F_{\lambda_1} \otimes \F_{\lambda_2}} \in \Hom_\CCC(\F_{\lambda_1} \otimes \F_{\lambda_2},\F_{\lambda_1+ \lambda_2})$ and $u|_{\F_0} \in \Hom_\C(\F_0,\F_0)$, respectively. Because, $V_L$ is a simple algebra object, the complex numbers $\mu_{\lambda_1,\lambda_2}$ have to be non-zero for all $\lambda_1,\lambda_2 \in L$. The basic properties of $\mu$ given by Definition \ref{algebraobjectdef} make the map
\begin{align*}
k^{\mu}: \; &L \times L \longrightarrow \mathbb{C}^{\times} \quad \; , \\
&(\lambda_1,\lambda_2) \mapsto \mu_{\lambda_1,\lambda_2}
\end{align*}
a \textit{normalized} 2-cocyle in the group cohomology set $\text{H}^2(L;\mathbb{C}^{\times})$\footnote{Here, $\CC^\times$ should be seen as a trivial $L$-module.} that also satisfies 
\begin{equation} \label{cocyclebraiding}
k^\mu(\lambda_1,\lambda_2) = \mu_{\lambda_1,\lambda_2} = e^{\pi i \lambda_1 \lambda_2}\mu_{\lambda_2,\lambda_1} = e^{\pi i \lambda_1 \lambda_2} \cdot k^\mu(\lambda_2,\lambda_1) \; .
\end{equation}
By \cite[Theorem 4.5]{DF}, the relation \eqref{cocyclebraiding} between $k^{\mu}$ and the braiding fixes the cohomology class of $k^\mu$. In particular, $k^\mu$ is in the same cohomology class as the trivial 2-cocycle $(\lambda_1,\lambda_2) \mapsto 1$. As constructing algebra objects with cohomologous normalised 2-cocycles satisfying \eqref{cocyclebraiding} is equivalent, we will henceforth assume that $k^{\mu}$ is the trivial cocycle so that $\mu_{\lambda_1,\lambda_2} = 1$ for all $\lambda_1,\lambda_2 \in L$.\\

Following \cite[Theorem 3.14]{CKL} and \cite[Theorem 3.65]{CKM}, we expect that the category of generalised modules for $V_L$ is braided equivalent to the category $\Rep^0 V_L$. However, since $\dCCC$ is semisimple, $\text{Rep}^0V_L$ is semisimple and is therefore equivalent to the category of ordinary (non-generalized) modules. That is,
\[ \text{Mod}V_L \; \cong \; \Rep^0V_L \; , \]
as braided monoidal categories. We will now construct the category $\text{Rep}^0V_L$ and compare its structure to the known module category of the VOA $V_L$.
%
%
\begin{proposition} \label{simpleobjectslatticealgebra}
The distinct isomorphism classes of simple objects in $\Rep^0 V_L$ are given by
\begin{align*}
\left\{ \mathscr{F}(\F_{\frac{x}{\sqrt{2N}}}) \; \big| \; x \in \{0,...,2N-1\} \right\}
\end{align*}

where $\ind: \dCCC \to \Rep V_L$ is the induction functor from Definition \ref{indfunctordef}.
\end{proposition}

\begin{proof}
By \cite[Theorem 4.5]{CKM}, every simple object is induced by a simple object. Therefore, it is enough to determine which $\F_{\lambda}$ induce to $\Rep^0V_L$. By semisimplicity of $\dCCC$, we can consider only its simple objects which coincide with the simple objects in $\C$. By \cite[Theorem 3.15]{CKL}, for any simple $\F_{\lambda} \in \C_{\oplus}$, $\mathscr{F}(\F_{\lambda}) \in \Rep^0V_L$ if and only if $M_{J,\F_{\lambda}} = \Id_{J \otimes \F_{\lambda}}$ where $M_{A,B}=c_{B,A} \circ c_{A,B}$ is the monodromy and $J=\F_{\sqrt{2N}}$ is the simple current from which $V_L$ is built. Recall that the twist on $\dCCC$ coincides with the twist on $\C$ for the $\F_{\lambda}$ and is given by $\theta_{\F_{\lambda}}=e^{\pi i \lambda^2} \Id_{\F_{\lambda}}$. By the balancing axioms we see that
\begin{align*}
\theta_{\F_{\sqrt{2N}} \otimes \F_{\lambda}} &= M_{\F_{\sqrt{2N}},\F_{\lambda}} \circ (\theta_{F_{\sqrt{2N}}} \otimes \theta_{\F_{\lambda}}) \\
&=M_{\F_{\sqrt{2N}},\F_{\lambda}} \circ e^{2 \pi i N} e^{\pi i \lambda^2} \Id_{\F_{\sqrt{2N}} \otimes\F_{\lambda} }\\
&=e^{\pi i \lambda^2} M_{\F_{\sqrt{2N}},\F_{\lambda}}.
\end{align*}

Since $\F_{\sqrt{2N}} \otimes \F_{\lambda} = \F_{\lambda + \sqrt{2N}}$, we see that
\begin{align*}
\theta_{\F_{\sqrt{2N}} \otimes \F_{\lambda}} &= e^{\pi i (\sqrt{2N} + \lambda)^2} \Id_{\F_{\sqrt{2N}} \otimes \F_{\lambda}}\\
&=e^{\pi i (2\lambda \sqrt{2N} + \lambda^2)} \Id_{\F_{\sqrt{2N}} \otimes \F_{\lambda}}.
\end{align*}

Hence, $M_{\F_{\sqrt{2N}},\F_{\lambda}} = e^{\pi i \lambda \sqrt{2N}} \Id_{\F_{\sqrt{2N}} \otimes \F_{\lambda}}$, so $M_{\F_{\sqrt{2N}},\F_{\lambda}} = \Id_{\F_{\sqrt{2N}} \otimes \F_{\lambda}}$ $\; \Leftrightarrow \;$ $\lambda \sqrt{2N} \in \mathbb{Z}$, that is, if and only if $\lambda$ is in the dual lattice $L^*=\frac{1}{\sqrt{2N}}\mathbb{Z}$ of $L=\sqrt{2N}\mathbb{Z}$. Fix $x \in L^*$ and $\ell \in L$. It remains to be shown that
\[ \ind(\F_{x})= \bigoplus\limits_{\lambda_1 \in L} \F_{\lambda_1} \otimes \F_{x} \text{ \quad and \quad } \ind(\F_{x+\ell})=\bigoplus\limits_{\lambda_2 \in L} \F_{\lambda_2} \otimes \F_{x+\ell} \]
are isomorphic as objects of $\Rep^0 V_L$. Define a morphism
\begin{equation} \label{shift}
\shift^\ell = \Big(\big\{(\lambda_1\big\} \mapsto \{\lambda_1 - \ell\}, \{1\}_{\lambda_1 \in L}\Big) \in \Hom_{\dCCC}\big(\ind(\F_x),\ind(\F_{x+\ell})\big) \; .
\end{equation}
To show that $\shift^\ell$ intertwines the $V_L$-actions on these two induced modules (see Definition \ref{indfunctordef}), recall that $\dCCC$ has trivial associativity (because $\C$ has trivial associativity). Therefore, the actions are simply given by \eqref{multiplicationlattice} where $\mu_{\lambda_1,\lambda_2} = 1$ for all $\lambda_1,\lambda_2 \in L$. It follows that 
\begin{align*}
\shift^\ell \, \circ \; \mu = \Big(\big\{(\ell_A,\lambda_1)\big\} \mapsto \{\ell_A + \lambda_1 - \ell\}, \{1 \cdot 1\}_{(\ell_A,\lambda_1) \in L^2}\Big)  = \mu \, \circ \shift^\ell \, ,
\end{align*}

in $\Hom_{\dCCC}\big(V_L \otimes_{\dCCC} \ind(\F_x), \ind(\F_{x + \ell})\big)$. Obviously, $\shift^\ell$ is invertible with inverse $\shift^{-\ell}$ and we conclude that $\ind(\F_x) \cong \ind(\F_{x+\ell})$ in $\Rep^0 V_L$. 
\end{proof}

To recover the associativity, braiding and twist scalars for the $V_L$ within the framework of algebra objects, one has to explicit the tensor product $\otimes_{V_L}$ (see Definition \ref{tensorAdef}) of between tuples of simple objects as of Proposition \ref{simpleobjectslatticealgebra}. Let $x, y \in L^*$ and consider the tensor product
\begin{equation} \label{sampleprelatticetensorprod}
\ind(\F_x) \otimes_{\dCCC} \ind(\F_y) = \bigoplus_{(\lambda_1,\lambda_2) \in L^2} \F_{\lambda_1} \otimes \F_{x} \otimes \F_{\lambda_2} \otimes \F_{y} \, .
\end{equation}
The left and right multiplication maps $\ind(F_x) \otimes_{\dCCC} V_L \otimes_{\dCCC} \ind(F_y) \rightarrow \ind(F_x) \otimes_{\dCCC} \ind(F_y)$ defining the tensor product $\otimes_{V_L}$ (see Definition \ref{tensorAdef}) are:
\begin{align}
m^{\text{left}} &= \left(\big\{(\lambda_1,\ell_A,\lambda_2)\big\} \mapsto \big\{(\lambda_1 + \ell_A,\lambda_2)\big\},\left\{e^{\pi i \big((\lambda_1 + x) \cdot \ell_A\big)}\right\}_{(\lambda_1,\ell_A,\lambda_2) \in L^3} \right) , \label{leftAaction} \\
m^{\text{right}} &=  \left(\big\{(\lambda_1,\ell_A,\lambda_2)\big\} \mapsto \big\{(\lambda_1,\lambda_2 + \ell_A)\big\},\left\{1\right\}_{(\lambda_1,\ell_A,\lambda_2) \in L^3} \right) . \label{rightAaction}
\end{align}
Let $\lambda_1, \lambda_2, \tilde{\lambda}_1, \tilde{\lambda_2} \in L$ be such that $\lambda_1 + \lambda_2 = \tilde{\lambda}_1 + \tilde{\lambda}_2$ and set $\ell_A = \lambda_1 - \tilde{\lambda}_1 = \tilde{\lambda}_2 - \lambda_2$. The formulas~\eqref{leftAaction} and~\eqref{rightAaction} mean that in the quotient space $\ind(\F_x) \otimes_{V_L} \ind(\F_y)$, the following operations on the components of $\ind(\F_x) \otimes_{\dCCC} \ind(\F_y)$ given in~\eqref{sampleprelatticetensorprod} are the same:
\vspace{-0.5em}
\begin{equation} \label{linkcomponenttensA}
\begin{array}{l}
 \text{$\bullet$ multiplying $\F_{\lambda_1} \otimes \F_x \otimes \F_{\lambda_2} \otimes \F_y$ by $e^{\pi i \big((\tilde{\lambda}_1 + x) \cdot (\lambda_1-\tilde{\lambda}_1)\big)}$} ; \\
\text{$\bullet$ multiplying $\F_{\tilde{\lambda}_1} \otimes \F_x \otimes \F_{\tilde{\lambda}_2} \otimes \F_y$ by $1$}. \\
\end{array}
\end{equation}
It follows that the components $\F_{\lambda_1} \otimes \F_x \otimes \F_{\lambda_2} \otimes \F_y$ and $\F_{\tilde{\lambda}_1} \otimes \F_x \otimes \F_{\tilde{\lambda}_2} \otimes \F_y$ are redundant in $\ind(\F_x) \otimes_{V_L} \ind(\F_y)$. Consequently, we expect:
\begin{align}
\ind(\F_x) \otimes_{V_L} \ind(\F_y) &= \bigoplus_{(\lambda_1,\lambda_2) \in \frac{L^2}{\ker(+)}} \left(\F_{\lambda_1} \otimes \F_x \otimes \F_{\lambda_2} \otimes \F_y\right) \notag \\
&\cong \bigoplus_{t \in L} \left(\F_t \otimes \F_x \otimes \F_y\right)  \label{isolatticetensorprod} \\
&= \bigoplus_{t \in L} \left(\F_t \otimes \F_{x+y}\right) \notag \\
&= \ind\left(\F_{x+y}\right), \notag
\end{align}
where an isomorphism of $\Rep^0 V_L$ is needed at line~\eqref{isolatticetensorprod}. The next lemma addresses this matter:

\begin{lemma} \label{fmap} The map of \eqref{isolatticetensorprod} can be defined as follows:
\begin{equation*} 
f^{x,y} = \left(\big\{(\lambda_1,\lambda_2)\big\} \mapsto \{\lambda_1 + \lambda_2\}, \{e^{\pi i \, x \lambda_2}\}_{(\lambda_1,\lambda_2) \in \frac{L^2}{\ker(+)} = L}\right) 
\end{equation*}

In particular, it is a well defined isomorphism in $\Rep^0V_L$ between $\ind(\F_x) \otimes_{V_L} \ind(\F_y)$ and $\ind(\F_{x+y})$. Note that it corresponds to the map $\ind \circ \boxtimes \rightarrow \boxtimes_{V_L} \circ (\ind \times \ind)$ given in \cite[Theorem 2.59 (2)]{CKM}.
\end{lemma}

\begin{proof} Following \cite[Proposition 5.7]{DF}, this is rather direct. However, the discussion leading to line~\eqref{isolatticetensorprod} gives a more conceptual approach. On a component $\F_{\lambda_1} \otimes \F_x \otimes \F_{\lambda_2} \otimes \F_{y}$ of $\ind(\F_x) \otimes_{V_L} \ind(\F_y)$, $f$ just braids $\F_x$ with $\F_{\lambda_2}$ and multiplies $\F_{\lambda_1} \otimes \F_{\lambda_2}$. To show that $f$ is a well-defined map, we must compare its effect on two components:
$$ \F_{\lambda_1} \otimes \F_x \otimes \F_{\lambda_2} \otimes \F_y \qquad \text{ and } \qquad \F_{\tilde{\lambda}_1} \otimes \F_x \otimes \F_{\tilde{\lambda}_2} \otimes \F_x $$
where $\lambda_1 + \lambda_2 = \tilde{\lambda}_1 + \tilde{\lambda}_2$.  As the above components are subject to the equivalence of line~\eqref{linkcomponenttensA}, the map $f^{x,y}$ is well defined if and only if
\begin{eqnarray*}
 & & f^{x,y}_{(\lambda_1,\lambda_2)} \cdot e^{\pi i \big((\tilde{\lambda}_1 + x) \cdot (\lambda_1-\tilde{\lambda}_1)\big)} = f^{x,y}_{(\tilde{\lambda}_1,\tilde{\lambda}_2)} \cdot 1 \\
 \Longleftrightarrow & & e^{\pi i (x \cdot \lambda_2)} \; = \; e^{\pi i (x \cdot \tilde{\lambda}_2)} \; e^{\pi i \big((\tilde{\lambda}_1 + x) \cdot (\lambda_1-\tilde{\lambda}_1)\big)} \\
 \Longleftrightarrow & & 1 = e^{\pi i \big(\tilde{\lambda}_1 \cdot (\lambda_1 - \tilde{\lambda})\big)} \; ,
\end{eqnarray*}
which holds since $\lambda_1,\tilde{\lambda}_1 \in L = \sqrt{2N}\ZZ$. It remains to be shown that $f^{x,y}$ is a morphism in $\Rep^0 V_L$ and that it is both injective and surjective. The morphism $f^{x,y}$ is in $\Rep^0 V_L$ since
\begin{equation*}
f^{x,y} \circ \mu = \left(\Big\{\big(\ell_A,(\lambda_1,\lambda_2)\big)\Big\} \mapsto \big\{\ell_A + \lambda_1 + \lambda_2\big\}, \{e^{\pi i \, x \lambda_2}\}_{\big(\ell_A,(\lambda_1,\lambda_2)\big) \in L \times \frac{L^2}{\ker(+)}}\right) = \mu \circ f^{x,y}
\end{equation*} 
Finally, the injectivity and surjectivity of $f^{x,y}$ follow from the facts that the component maps of $f^{x,y}$ are isomorphisms at the level of Fock spaces (non-zero scalars) and that
\begin{align*}
+ : \;\; L^2/\ker(+) \quad &\stackrel{1:1}{\longleftrightarrow} \quad L 
\end{align*}
is a natural bijection between the index sets of the domain \& codomain of $f$.
\end{proof}

We conclude by the following:

\begin{theorem}
The skeleton of $\text{Rep}^0 V_L$ with simple objects given in \eqref{simpleobjectslatticealgebra} is a rigid monoidal category with tensor product
\begin{equation*} 
\ind(\F_x) \otimes_{V_L} \ind(\F_y) = \ind(\F_{x+y-k(x,y)}).
\end{equation*}

where $k:L^*/L \times L^*/L \to L$ is defined in \eqref{cocyclesimplecurr}. This category has associativity constraint, braiding, and twist given by
\begin{align*}
a^{V_L}_{\ind(\F_x),\ind(\F_y),\ind(\F_z)} &= (-1)^{x \cdot k(y,z)} \Id_{\ind(\F_{x+y+z-k(x,y)-k(x+y,z)})} \\
c_{\ind(\F_x),\ind( \F_y)}^{V_L} &= e^{\pi i xy}\Id_{\ind(\F_{x+y-k(x,y)})}\\
\theta^{V_L}_{\ind(\F_x)} &= e^{\pi i x^2} \Id_{\ind(\F_x)}
\end{align*}

Notice that the above associativity, braiding and twist scalars match their respective analogues for $V_L$ seen as a VOA (see subsection 4.1).

\end{theorem}

\begin{proof} The associativity and braiding scalars are also computed in \cite[Proposition A.2]{DF} where $H = L^*$, $K = L$, $\alpha = k^\mu = 1$ and $\theta = k$. However we include the analogous computations in this proof in order to complete our alternate approach. Let $\dCCC^0$ be the full subcategory of objects of $\dCCC$ which induces to $\text{Rep}^0V_L$. By Proposition \ref{C0induction}, the induction functor $\ind : \mathcal{C}_{\oplus}^0 \rightarrow \text{Rep}^0V_L$ is a braided tensor functor and by Proposition \ref{simpleobjectslatticealgebra}, every simple object in $\text{Rep}^0V_L$ is induced by $\ind$.  The tensor product on simple modules is given by
\begin{equation} \label{detailslatticeskeletaltensorprod}
\tilde{f}: \; \ind(\F_x) \otimes_{V_L} \ind(\F_y) \stackrel{\eqref{isolatticetensorprod}}{\cong} \ind(\F_x \otimes \F_y) \stackrel{\eqref{shift}}{=} \ind(\F_{x+y-k(x,y)})
\end{equation}
for any $\ind(\F_x),\ind(\F_y)$ in the set of Proposition \ref{simpleobjectslatticealgebra}. Notice that the map $\tilde{f}$ is the map $f$ from Lemma \ref{fmap} accompanied by a shift in the index. The category is rigid  by Proposition 2.77 and Lemma 2.78 in \cite{CKM}. By Proposition 2.67 of the same reference, the braiding $c_{-,-}^{V_L}$ satisfies 
\[ \Id_{V_L} \otimes c_{\F_x,\F_y} = \ind(c_{\F_x,\F_y}) = \tilde{f}^{-1}_{\F_x,\F_y} \circ c_{\ind(\F_x),\ind(\F_y)}^{V_L} \circ \tilde{f}_{\F_x,\F_y}. \]

The braiding $c_{\F_x,\F_y}$ on $\mathcal{C}$ is given by $c_{\F_x,\F_y}=e^{ \pi i xy} \Id_{\F_x \otimes \F_y}$, so we see that
\[c_{\ind(\F_x),\ind(\F_y)}^{V_L}= e^{\pi i xy}\Id_{\ind(\F_{x+y}}). \]

Recall that $\theta_{\F_{\lambda}^k}=e^{\pi i k^2\lambda^2} \Id_{\F_{\lambda}^k} = e^{\pi i k^2 2N m^2} \Id_{\F_{\lambda}^k} = \Id_{\F_{\lambda}^k}$, so by \cite[Corollary 2.89]{CKM} 
\[ \theta_{\ind(\F_x)} = \ind(\theta_{\F_x}) = e^{\pi i x^2} \Id_{\ind(\F_x)}. \]

To compute the associativity, consider three objects $\ind(\F_x),\ind(\F_y),\ind(\F_z)$ and for each pair among them, the corresponding tensor product isomorphisms \ref{detailslatticeskeletaltensorprod}. By \cite[Theorem 2.59]{CKM}, the associativity map $a^{V_L}_{\ind(\F_x),\ind(\F_y),\ind(\F_z)}$ satisfies:
\begin{align*}
&a^{V_L}_{\ind(\F_x),\ind(\F_y),\ind(\F_z)} \circ (\Id_{\ind(\F_x)} \otimes_{V_L} \tilde{f}_{\ind(\F_y),\ind(\F_z)}) \circ \tilde{f}_{\ind(\F_x),\ind(\F_{y+z - k(y,z)})} \\
= ( &\tilde{f}_{\ind(\F_x),\ind(\F_y)} \otimes_{V_L} \Id_{\ind(\F_z)} ) \circ \tilde{f}_{\ind(\F_{x+y - k(x,y)}),\ind(\F_z)} \circ \ind(a_{\F_x,\F_y,\F_z}) 
\end{align*}
Let $m_{x,y}=e^{\pi i xy}$ be the scalar associated to the $f$ map of Lemma \ref{fmap}. Then, the associativity $a^{V_L}_{\ind(\F_x),\ind(\F_y),\ind(\F_z)}$ maps an arbitrary component as follows: 

\begin{align*}
(\F_{\lambda_1} \otimes \F_x \otimes \F_{\lambda_2} \otimes \F_y) \otimes \F_{\lambda_3} \otimes \F_z &\cong (\F_{\lambda_1 + \lambda_2} \otimes \F_{x+y-k(x,y)} ) \otimes \F_{\lambda_3} \otimes \F_{z} &  &m_{x,\lambda_2}\\
 & \cong \F_{\lambda_1 + \lambda_2 + \lambda_3} \otimes \F_{x+y+z-k(y,z)-k(x,y+z)}  & &m_{x+y-k(x,y),\lambda_3} \\
& = \F_{\lambda_1 + \lambda_2 + \lambda_3} \otimes \F_{x+y+z-k(x,y)-k(x+y,z)}& &\;\\
& \cong \F_{\lambda_1} \otimes \F_x  \otimes \F_{\lambda_2+\lambda_3} \otimes \F_{y+z-k(x,y)} & &m_{x,\lambda_2+\lambda_3}^{-1}\\
&\cong \F_{\lambda_1} \otimes \F_x \otimes (\F_{\lambda_2} \otimes \F_y \otimes \F_{\lambda_3} \otimes \F_z)& &m_{y,\lambda_3}^{-1}
\end{align*}

The corresponding scalar is equal to
\begin{align*} \label{scalarassoc}
e^{\pi i \big( x \lambda_2 + \big(x + y -k(x,y)\big)\lambda_3 - x\big(\lambda_2+\lambda_3-k(y,z)\big) - y \lambda_3 \big)} &= e^{\pi i \big(-k(x,y)\lambda_3 + x \cdot k(y,z)\big)} \\
&= e^{\pi i \, x \cdot k(y,z)} \quad \text{since } k(x,y),\lambda_3 \in L = \sqrt{2N} \ZZ \\
&= (-1)^{x \cdot k(y,z)} \quad \text{since } x \in L^* \text{ \& } k(y,z) \in L
\end{align*} 
As the scalar $(-1)^{x \cdot k(y,z)}$ is independent of $\lambda_1, \lambda_2, \lambda_3 \in L$, we conclude that 
$$ a^{V_L}_{\ind(\F_x),\ind(\F_y),\ind(\F_z)} = (-1)^{x \cdot k(y,z)} \Id_{\ind(\F_{x+y+z-k(x,y)-k(x+y,z)})} $$
as desired.\\
\end{proof}

\addcontentsline{toc}{section}{References}


\begin{thebibliography}{CGP00}

\bibitem[A]{A} D. Adamovi\'c, A construction of admissible {$A^{(1)}_1$}-modules of level {$-\frac 43$}, J. Pure Appl. Algebra vol.196 2-3 (2005), 119--134.


\bibitem[AKMPP]{AKMPP} D. Adamovi\'c, V. Kac, P. M\"oseneder Frajria, P. Papi, O. Per\v se, Finite vs. infinite decompositions in conformal embeddings, Comm. Math. Phys. vol. 348 2 (2016), 445--473.


\bibitem[AGV]{AGV} M. Artin, A. Grothendieck, J.L. Verdier, Th{\'e}orie des topos et cohomologie
{\'e}tale des sch{\'e}mas I, II, III, Lecture Notes in Mathematics, vol. 269, 270, 305, 
Springer, 1971.

\bibitem[ACR]{ACR} J. Auger, T. Creutzig, D. Ridout, Modularity of logarithmic parafermion vertex algebras, arXiv:1704.05168.

\bibitem[ACKR]{ACKR} J. Auger, T. Creutzig, S. Kanade, M. Rupert, Semisimplification of a Category of Modules for the Logarithmic $\mathcal{B}_p$-Algebras, in preparation.

\bibitem[C]{C} T. Creutzig, W-Algebras for Argyres-Douglas Theories, Eur. J. Math. 3 (2017), 659-690.

\bibitem[CGR]{CGR} T. Creutzig, A.M. Gainutdinov, I. Runkel, A quasi-Hopf algebra for the triplet vertex operator algebra, in preparation.



\bibitem[CKL]{CKL} T. Creutzig, S. Kanade, A. Linshaw,  Simple Current Extensions Beyond Semi-Simplicity, arXiv:1511.08754.

\bibitem[CKLR]{CKLR} T. Creutzig, S. Kanade, A. Linshaw, D. Ridout, Schur-Weyl Duality for Heisenberg Cosets, arXiv:1611.00305.

\bibitem[CKM]{CKM} T. Creutzig, S. Kanade, R. McRae, Tensor categories for vertex operator superalgebra extensions, arXiv:1705.05017.

\bibitem[CRW]{CRW} T.Creutzig, D. Ridout, S. Wood, Coset constructions of logarithmic $(1, p)$-models, Letters in Mathematical Physics 104 5 (2014), 553-583.


\bibitem[D]{D} C. Dong, Vertex Algebras Associated with Even Lattices, Journal of Algebra 161 1 (1993), 245-265.


\bibitem[DL]{DL} C. Dong, J. Lepowsky, Generalized vertex algebras and relative vertex operators, Progress in Mathematics, Birkh\"auser Boston, Inc., Boston, Vol. 112, 1993.

\bibitem[DLM]{DLM} C. Dong, H. Li, and G. Mason, Simple current extensions of vertex operator algebras, Comm. Math. Phys. 180 3 (1996), 671-707.

\bibitem[DF]{DF} A. Davydov, V. Futorny, Commutative Algebras in Drinfeld Categories of Abelian Lie Algebras, Proceedings of the Edinburgh Mathematical Society 55 (2012), 613-633.

\bibitem[EGNO]{EGNO} P. Etingof, S. Gelaki, D. Nikshych, V. Ostrik, Tensor Categories, American Mathematical Society, Mathematical Surverys and Monographs, Vol. 205, 2015.


\bibitem[FB]{BF} E. Frenkel, D. Ben-Zvi, Vertex Algebras and Algebraic Curves, American Mathematical Society, Mathematical Surveys and Monographs, Vol. 88, 2001.

\bibitem[FRS]{FRS} J. Fuchs, I. Runkel, C. Schweigert, TFT construction of RCFT correlators III: Simple Currents, Nucl.Phys. B694 (2004), 277-353.

\bibitem[HKL]{HKL} Y.-Z. Huang, A. Kirillov Jr. J. Lepowsky, Braided tensor categories and extensions of vertex operator algebras.
Comm. Math. Phys. 337 3 (2015), 1143-1159.

\bibitem[HL]{HL} Y.-Z. Huang, J. Lepowsky, Tensor products of modules for a vertex operator algebra and vertex tensor categories, Lie Theory and Geometry, in honor of Bertram Kostant, ed. R. Brylinski, J.-L. Brylinski, V. Guillemin, V. Kac, Birkh$\ddot{\text{a}}$user, Boston, 1994, 349-383.

\bibitem[HLZ1]{HLZ1} Y.-Z. Huang, J. Lepowsky and L. Zhang, Logarithmic tensor category theory for generalized
modules for a conformal vertex algebra, I: Introduction and strongly graded algebras and their
generalized modules, Conformal Field Theories and Tensor Categories, Proceedings of a Workshop
Held at Beijing International Center for Mathematics Research, ed. C. Bai, J. Fuchs, Y.-Z.
Huang, L. Kong, I. Runkel and C. Schweigert, Mathematical Lectures from Beijing University,
Vol. 2 (2014), 169-248.

\bibitem[HLZ2]{HLZ2} Y.-Z. Huang, J. Lepowsky and L. Zhang, Logarithmic tensor category theory for generalized
modules for a conformal vertex algebra, II: Logarithmic formal calculus and properties of logarithmic
intertwining operators, arXiv:1012.4196.

\bibitem[HLZ3]{HLZ3} Y.-Z. Huang, J. Lepowsky and L. Zhang, Logarithmic tensor category theory for generalized
modules for a conformal vertex algebra, III: Intertwining maps and tensor product bifunctors ,
arXiv:1012.4197.

\bibitem[HLZ4]{HLZ4} Y.-Z. Huang, J. Lepowsky and L. Zhang, Logarithmic tensor category theory for generalized
modules for a conformal vertex algebra, IV: Constructions of tensor product bifunctors and the
compatibility conditions , arXiv:1012.4198.

\bibitem[HLZ5]{HLZ5} Y.-Z. Huang, J. Lepowsky and L. Zhang, Logarithmic tensor category theory for generalized
modules for a conformal vertex algebra, V: Convergence condition for intertwining maps and
the corresponding compatibility condition, arXiv:1012.4199.

\bibitem[HLZ6]{HLZ6} Y.-Z. Huang, J. Lepowsky and L. Zhang, Logarithmic tensor category theory for generalized
modules for a conformal vertex algebra, VI: Expansion condition, associativity of logarithmic
intertwining operators, and the associativity isomorphisms, arXiv:1012.4202.

\bibitem[HLZ7]{HLZ7} Y.-Z. Huang, J. Lepowsky and L. Zhang, Logarithmic tensor category theory for generalized
modules for a conformal vertex algebra, VII: Convergence and extension properties and applications
to expansion for intertwining maps, arXiv:1110.1929.

\bibitem[HLZ8]{HLZ8} Y.-Z. Huang, J. Lepowsky and L. Zhang, Logarithmic tensor category theory for generalized
modules for a conformal vertex algebra, VIII: Braided tensor category structure on categories
of generalized modules for a conformal vertex algebra, arXiv:1110.1931.

\bibitem[KMPX]{KMPX} V. Kac, P. M\"oseneder Frajria, P. Papi, F. Xu, Conformal embeddings and simple current extensions, Int. Math. Res. Not. IMRN 14 (2015), 5229-5288.

\bibitem[KO]{KO} A. Kirillov Jr., V. Ostrik, On a q-analogue of the McKay correspondence and the ADE classification of sl2 conformal
field theories. Adv. Math. 171 (2002), no. 2, 183-227. 

\bibitem[KW]{KW} V. Kac and M. Wakimoto, Integrable highest weight modules over affine superalgebras and Appell’s function, Comm. Math. Phys. 215 (2001) 631-682.

\bibitem[La]{La}  C. H. Lam, Induced modules for orbifold vertex operator algebras. J. Math. Soc. Japan 53 3 (2001), 541-557.

\bibitem[LaLaY]{LaLaY} C. H. Lam, N. Lam, H. Yamauchi, Extension of unitary Virasoro vertex operator algebra by a simple module. Int. Math. Res. Not. 11 (2003), 577–611.

\bibitem[LL]{LL} J. Lepowsky, H. Li, Introduction to Vertex Operator Algebras and Their Representations. Birkhäuser, Boston (2003).

\bibitem[P]{P} B. Pareigis, On Braiding and Dyslexia, Journal of Algebra, Vol. 171 2 (1995),  413-425.
 
\bibitem[PP]{PP} N. Popescu and L. Popescu, Theory of Categories, Editura Academiei, 1979.

\bibitem[Y]{Y} H. Yamauchi, Module categories of simple current extensions of vertex operator algebras, Journal of Pure and Applied Algebra Volume 189, Issues 1–3 (2004), 315-328.




\end{thebibliography}
 \end{document}